\documentclass[10pt,twoside]{article}
\usepackage{amssymb,amsthm}
\usepackage[mathscr]{eucal}
\usepackage{eufrak}
\usepackage{amsmath}
\usepackage{mathrsfs}
\usepackage{hyperref}
\font\bg=cmbx10 scaled 1200
\def\OO{\Omega}

\def\intll#1#2{\int\limits_{#1}^{#2}}

\def\O{\Omega}

\def\R{\Bbb R}
\def\N{\Bbb N}
\def\o{\"{o}}
\def\à{\`{a}}
\def\è{\`{e}}
\def\ì{\`{i}}
\def\ù{\`{u}}
\def\ò{\`{o}}
\def\é{\'{e}}

\def\dy{\displaystyle}
\def\ve{\varepsilon}

\def\pa{\partial}
\def\be{\begin{equation}}
\def\ba{\begin{array}}
\def\ea{\end{array}}
\def\ee{\end{equation}}
\def\vs1{\vspace{1ex}}
\def\vp{\varphi}
\def\ov{\overline}
\def\po{\partial\Omega}
\def\sex{{s_{\rm ex}}}
\font\sc=cmcsc10
\title{\bg Singular \mbox{\large$p$}-Laplacian parabolic system in exterior domains: higher regularity of
solutions and related properties of extinction and asymptotic behavior in time.
}
\author{\sc Francesca Crispo
\and\sc Carlo Romano Grisanti
\and\sc Paolo Maremonti
}

\date{}

\begin{document}

\maketitle
\noindent{\bf Abstract}  - 
{\small
We consider the IBVP in exterior domains for the $p$-Laplacian parabolic system. We prove regularity up to the boundary, extinction properties for $p\in\left(\frac{2n}{n+2},\frac{2n}{n+1}\right)$ and exponential decay for $p=\frac{2n}{n+1}$.}
 \par\noindent
 \vskip.1cm
{\par\noindent{\bf Keywords:}\ $p$-Laplacian, regularity, extinction and asymptotic behavior of the solution.}
 \vskip.1cm
{\par\noindent{\bf Mathematics Subject Classification:} 35K92,\ 35B65,\ 35B40.}
 \vskip -0.7true cm\noindent
\newcommand{\red}{\protect\bf}
\renewcommand\refname{\centerline
{\red {\normalsize \bf References}}}
\newtheorem{theorem}{Theorem}[section]
\newtheorem{lemma}[theorem]{Lemma}
\newtheorem{corollary}[theorem]{Corollary}
\newtheorem{remark}[theorem]{Remark}
\newtheorem{definition}[theorem]{Definition}
\newtheorem{proposition}[theorem]{Proposition}
\newtheorem{inequality}[theorem]{Inequality}
\numberwithin{equation}{section}
\section{Introduction}

The $p-$Laplace equation is a prototype example of non linear PDE. We consider the parabolic singular case $1<p<2$ for vector valued functions, namely
 \be\label{PF}\begin{array}{ll}\dy
u_t-\nabla\cdot\left(|\nabla u|^{p-2}\nabla u\right)\dy= 0\,,&\hskip-0.2cm\textrm{
in }(0,T)\times\O,
\\\dy \hskip2.4cmu(t,x)\dy=0\,, &\hskip-0.2cm \textrm{ on }
(0,T)\times\po,\;\\\dy\hskip2.4cm u(0,x)=u_\circ(x),&\hskip-0.2cm\mbox{ on
}\{0\}\times\O.\end{array}\ee
where $\O$ is a bounded or exterior $C^2$ domain of $\R^n$ and $u:\O\longrightarrow\R^N$ a vector valued function,
with $n\ge2$ and $N\ge1$. \par Problem \eqref{PF} is widely studied in the case of bounded domains $\O$ and in the case of the Cauchy problem. We would like to say that, in the case of $\O$ bounded, the literature can be split  in two branches. A former is a classical theory which is essentially  devoted to the analysis of the H\o lder's regularity of the gradient of weak solutions, see \cite{AMS},  \cite{BDM},  \cite{CDB},  \cite{choe1},  \cite{choe},  \cite{DB},  \cite{DB1},  \cite{DBF1},  \cite{DBF},  \cite{DMSt},  \cite{LMV},  \cite{misawa},  
\cite{misawa1},  \cite{Miz}, \cite{Sch1}, \cite{W}. The latter is more recent and it is based on the local or global $W^{2,q}$-regularity for suitable exponents $q$,  see \cite{AMS}, \cite{BDV}, \cite{CM}, \cite{LVP}, \cite{Sch2}. In this connection it is important to point out that only in \cite{CM} is obtained the $L^\infty(0, T; W^{2,q}(\O))$ regularity up to the boundary with an exponent $q\geq 2$. It is deduced by the aid of the results in \cite{CBDV} and  \cite{CM2014} related to the boundary value problem associated to the elliptic case.  On the other hand if we exclude the special case of the Cauchy problem, the initial boundary value problem in unbounded domains appears overlooked. The same is for the boundary value problem associated to the steady  equations. The last problem very recently has received contributions  for the  elliptic problem and for a perturbed elliptic problem \cite{CGM,CMmod}.\par The aim of this paper is to fill the gap of results between the cases of the IBVP for $\O$ bounded and IBVP for $\O$ exterior domain.  Particular regards are posed to the questions of the regularity and extinction properties of the solutions.\par This paper is the natural evolution
%The objective of the present paper is to investigate the integrability of the second derivatives, especially on exterior domains, even if we achieve new results also in the case of bounded domains. 
%{\bf breve storia del problema}
    of a project, concerning the regularity of the $p$-Laplace system, whose previous chapters are the papers \cite{CM,CGM}. The former deals with the parabolic problem on bounded domains and the latter concerns the elliptic system on exterior domains.
A common feature of the high integrability results in \cite{CM,CGM} (likewise \cite{CBDV,CM2014, CMmod}) is the connection between the power $q$ of summability of the second spatial derivatives and the exponent $p$ which describes the singularity of the operator: as $q$ increases, $p$ must approach 2 from below. Roughy speaking, in the scalar parabolic case, the second derivatives become more integrable as the equation get closer to the heat equation. Together with this constraint, even for bounded domains (see \cite{CM}) we can find other restrictions on $p$ which sound to be more technical than intrinsic to the problem. To get rid of some of these, we refine the duality method exploited in the quoted paper, resorting to a further adjoint problem. The result is obtained for a bounded domain and extended to the case of an exterior one. Our technique allows us also to push upward the exponent of integrability of $D^2u$. In this respect we remark that we obtain a power that is higher than the space dimension, achieving the H\"older continuity of $\nabla u$ up to the boundary, even for an exterior domain.
\par
We like to point out that the special issue about the square summability of $D^2u$ deserves a particular consideration, since the result becomes very clean requiring simply $p>\frac{2n}{n+2}$.
\par
We want to remark that we do not analyze the regularity of the solution, instead we exhibit the existence of a regular solution and we use its uniqueness.
\par In order to tackle the mathematical question related to the extinction of the solutions, we need a $L^s$-theory for $s\in(1,2)$. In this respect we point out that the result of uniqueness holds with the stronger hypothesis of initial data in $L^s\cap L^2$. We would like to remark that we cannot omit the $L^2$ assumption on $u_\circ$. Actually the difficulties are related with the non-linear character of the system and the weakness of the $L^s$-theory for $s<2$. However the same difficulties are met in the IBVP on bounded domains.
The character of unbounded domains and the non-linearity of the $p$-laplacian give a special interest to the technique and to the results. Among the results, we obtain the following generalized energy relation:
\be\label{GER}\| u(t)\|_s^2+c\intll0t\| \nabla u(\tau)\|^p_{\frac{sp}2}\leq \| u_\circ\|_s^2,\mbox{ for all }t>0,\ee where $c$ is independent of $u_\circ$. The above generalized energy inequality assumes a particular interest even in the case of linear parabolic systems. Actually, for the following IBVP
 \be\label{LPP}\begin{array}{rll}\dy
u_t-\Delta u\dy &\hskip-.2cm= 0\,,&\textrm{
in }(0,T)\times\O,
\\\dy u(t,x)\dy&\hskip-.2cm=0\,, & \textrm{ on }(0,T)\times\po,\;\\
\dy u(0,x)&\hskip-.2cm=u_\circ(x),&\mbox{ on
}\{0\}\times\O.\end{array}\ee
 it is well known that the energy equality
\be\label{GERLPS}\| u(t)\|_2^2+2\intll0t\| \nabla u(\tau)\|_2^{ 2}= \| u_\circ\|_2^2\ee
holds for any $t>0$. In the case of a $L^q$-theory, $q\in(1,n]$, the above relation is replaced by estimates of the kind \be\label{GELqN}\|\nabla u(t)\|_q\leq c(t-t_0)^{-\frac12}\|u(t_0)\|_q,\mbox{ for all }t>t_0\geq 0.\ee
It is evident that (\ref{GELqN}) cannot imply 
$$\intll{t_1}t\|\nabla u(\tau)\|_q^2d\tau\leq c\|u_\circ\|^2_q,\mbox{ for all }t>t_1>t_0,$$
but it can only furnish  the weaker property  $\|\nabla u\|_q\in L^2_w(t_0,\infty)$, where $L^2_w$ is the Lorentz space. Hence estimate (\ref{GER}) has a special interest in the case of $p=2$ (linear case), because it reproduces for all $s\in(1,2]$ a property that was relegated only to the $L^2$-theory. 
\par The following theorems are proved in Sections \ref{L2}, \ref{Lq} and \ref{extinction}.
\begin{theorem}\label{W22IN}
Let be $p\in\left(\frac{2n}{n+2},2\right)$, $\O$ a bounded or exterior $C^2$ domain of $\R^n,\ n\ge2$ and $u_\circ\in L^2(\O)$. Then, for any $\ve>0,\ D^2u\in L^\infty(\ve,T;L^2(\O))$, where $u$ is the unique solution $u$ of \eqref{PF} and
\be\label{D2L2Omega}\|D^2 u(t)\|_2\le\frac c{t^{\frac{1+\ov\gamma}{p-1}}}\|u_\circ\|_2^{\frac{(2-p)\ov\gamma+1}{p-1}}+\frac c{t^{\frac1p}}\|u_\circ\|_2^{\frac2p}\ee
with $\ov\gamma=\frac{(n-2)(2-p)}{p(n+2)-2n}$ if $n\ge 3$ or $\ov\gamma=\frac{r-2}{r(p-2)}$ for any $r\in\left(2,\frac2{p-1}\right)\cap\left[2,2+\frac{4(p-1)}{(2-p)^2}\right]$ if $n=2$. Moreover, for any $q\in\left[2,2+\frac{4(p-1)}{(2-p)^2}\right]$ and $\ve>0$ we have that $u_t\in L^\infty(\ve,T;L^q(\O))$ and
\be\label{Lqu_t}\|u_t(t)\|_{q}\leq \frac{c}{t}{\null_{1+\gamma}}\,
\|u_\circ\|_2^{(2-p)\gamma+1},\, \mbox{ a.e. in }(0,T),\ee
%\be\label{Lqu_t}t^{1+\gamma}u_t\in L^\infty(0,T;L^q(\O))\ee
with $\gamma=\gamma(q')$ given by \eqref{Maxpostraa}.
\end{theorem}

\begin{theorem}\label{W2qbounded}
Let $E\subset\R^n,\ n\ge2$ be a bounded $C^2$ domain, $p>\frac{2n}{n+2}$ and $q\in\left[2,2+\frac{4(p-1)}{(2-p)^2}\right]$.  Moreover, following Definition \ref{C(q)}, let
$$\phi(q)=\left\{\ba{ll}2-\frac1{\ov C(q)}&\mbox{ if }q\not= n\\
\inf\limits_{q>n}\left\{2-\frac1{\ov C(q)}\right\}&\mbox{ if } q=n.\ea\right.$$
If $p>\max\{\phi(2),\phi(q)\}$ then, the unique solution of \eqref{PF} belongs to $L^\infty(\ve,T;W^{2,q}(E))$, for any $\ve>0$.
\end{theorem}

\begin{theorem}\label{W2qexterior}
Let $\O$ be an exterior $C^2$ domain of $\R^n,\ n\ge2$ and $p>\frac{2n}{n+2}$. For any $q\in\left[2,2+\frac{4(p-1)}{(2-p)^2}\right]$, there exists $\ov p(q)<2$ such that if $p\in(\ov p(q),2)$ and $u$ is the unique solution of \eqref{PF} then $D^2u\in L^\infty(\ve,T;L^q(\O))$.
\end{theorem}

\begin{theorem}\label{EXTIN}Let $\O$ be an exterior $C^2$ domain of $\R^n$. Assume $p\in (\frac{2n}{n+2},\frac{2n}{n+1})$ and $u_\circ\in L^{s_{\rm ex}}(\O)\cap L^s(\OO)$, with $s_{\rm ex}:=n(\frac2p-1)$ and $s>s_{\rm ex}$. Then there exists a solution $u$ of problem\,\eqref{PF}, in the sense of Definition \ref{weaksolutionLs}, which enjoys the extinction property
\be\label{EXT-I}u(t)=0\mbox{ for all }t\geq T_{\rm ex}\ee
where 
$$T_{\rm ex}\leq \frac c{2-p}\|u_\circ\|_\sex^2.$$ If $u_\circ\in L^{s_{\rm ex}}(\O)\cap L^2(\O)$, then the solution $u$ is unique. Moreover, if $p=\frac {2n}{n+1}$ and $u_\circ\in L^1(\O)\cap L^s(\O)$, $s\in (1,2]$, then we get the exponential decay
 \be\label{ESP-I}\|u(t)\|_2\leq \frac c{\ve^\gamma}\|u_\circ\|_s^\alpha e^{-c(t-\ve)\|u_\circ\|_1^{-1/(n+1)}},
 \mbox{ for all }t>\ve>0.\ee\end{theorem}
\par Theorem\,\ref{EXTIN} furnishes a result typical of the $p$-laplacian parabolic problem, that is the extinction of the solution in a finite time. This property depends on the nature of the domain $\O$ of the IBVP. For $\O$ bounded we refer to DiBenedetto \cite{DB}. The known result in the case of unbounded domains is related to the Cauchy problem see \cite{DB} and \cite{HV}. This case is characterized by the fact that the extinction of the solution holds with initial data belonging to $L^{\sex}(\O)$ with $\sex:=n(\frac2p-1)$. In Theorem\,\ref{EXTIN} we prove this kind of result for $p\in (\frac {2n}{n+2},\frac{2n}{n+1})$. It is important to stress that we need an $L^s$-theory $s<2$ of existence as a key tool in order prove the extinction. This is in harmony with the result of the Cauchy problem. In Theorem \ref{Lsexistence} we develop a $L^s$-theory of existence of solutions which are \underline{regular} for $t>0$.  However we are not able to prove uniqueness unless for initial data $u_\circ\in L^s(\O)\cap L^2(\O)$. We complete this kind of results by proving that in the case $p=\frac{2n}{n+1}$ the solutions admit an exponential decay in time.
 \par We complete the introduction furnishing a generalized energy inequality related to the solutions of the linear IBVP for parabolic systems \eqref{LPP}

\begin{theorem}\label{LPS}Let $\O$ be an exterior domain and $u_\circ\in L^\sigma(\O)$ with $\sigma\in(1,2]$. Then there exists a unique solution to problem \eqref{LPP} such that $u$ is smooth for $t>0$ and 
\be\label{LPS-I} \|u(t)\|_\sigma^2+2(\sigma-1)\int_0^t\|\nabla u(\tau)\|^2_\sigma d\tau \leq \|u_\circ\|_\sigma^2,\mbox{ for all }t>0.\ee
\end{theorem}

Theorem \ref{LPS} is proved in Section \ref{energyLq}.
\par
The plan of the paper is the following. In Section \ref{not} we introduce the notation, the function spaces, our notion of solution and some results concerning the elliptic problem. In Section \ref{exi} we quote the existence theorem for the parabolic problem on bounded domains furnishing the explicit estimates which are hidden in the original result; further we prove our existence theorem on exterior domains. Section \ref{adj} contains two adjoint parabolic problems which are used in Section \ref{ut} to estimate the time derivative in $L^q(\O)$ by duality. The integrability of the second spatial derivatives is investigated in Section \ref{L2} and Section \ref{Lq}, respectively in $L^2$ and $L^q$, using the elliptic results with $u_t$ acting as a force term. In Section \ref{C1alpha} we obtain the H\"older regularity of the gradient by Sobolev-Morrey embedding results. Section \ref{Lstheory} is entirely devoted to the existence theory with initial data in $L^s(\O)$. In Section \ref{extinction} we investigate the extinction and exponential decay of the solutions. Finally, in Section \ref{energyLq} we apply the methods of Section \ref{extinction} to prove the energy inequality in $L^s(\O)$, with $ 1<s<2$, for linear parabolic IBVP.
\bigskip
\par
{\sc Acknowledgments} - {\small The authors are grateful to the referee who pointed out three critical points in the proof.
\par 
This research is partially supported by MIUR via the PRIN 2015 ``Hyperbolic Systems of Conservations Laws and Fluid Dynamics: Analysis and Applications''. The research activity of F. Crispo and P. Maremonti is performed under the
auspices of National Group of Mathematical Physics (GNFM-INdAM). The research activity of C. R. Grisanti is performed under the
auspices of National Group of Mathematical Analysis, Probability and their Applications (GNAMPA-INdAM)}.
      
\section{Notation and preliminary results}\label{not}

We denote by $\O$ an exterior domain i.e. the complementary of a compact connected set of $\R^n$. In this context, we can find a real number $R_0>0$   such that $(\R^n\setminus\O)\subset B(0,R_0)$. On the other hand, we reserve the letter $E$ for bounded subsets of $\R^n$. In some statements the letter $\O$ is used at the same time for bounded or exterior domains and the occurrence is explicitly enhanced.\par
For any $R>0$ we define a smooth cut-off function $h_R:\R^n\longrightarrow[0,1]$ such that 
\be\label{hR}h_R(x)=\left\{\ba{ll}1&\mbox{if }|x|\le R\\
0&\mbox{if }|x|\ge2R,\ea\right.\qquad |\nabla h_R|\le \frac cR.\ee

Together with the usual Lebesgue, Sobolev and Bochner spaces we also make use of some other suitable spaces in the framework of exterior domains. First the space $\widehat W_0^{1,p}(\O)$ which is the completion of $C^\infty_0(\O)$ in the norm $|\vp|_{1,p}:=\|\nabla\vp\|_p$ and that, in the case of a bounded domain, coincides with $W^{1,p}_0(\O)$. We introduce also the Banach space $V(\O):=\widehat W^{1,p}_0(\O)\cap L^2(\O)$ and the Bochner space $V^{p,p'}(0,T;\O):=\{\psi\in L^p(0,T;V(\O)):\psi_t\in L^{p'}(0,T;V(\O)')\}$ with the norm $\|\psi\|:=\|\psi\|_{L^p(0,T;V(\O))}+\|\psi_t\|_{L^{p'}(0,T;V(\O)')}$ (see \cite[Sec. 23.6]{Zei}). The symbol $\langle\cdot,\cdot\rangle$ stands for the duality pairing between a Banach space and its dual.
\par
We begin with the definition of a quantity which is crucial in most of our results.
\begin{definition}\label{C(q)}
Let $E$ be a bounded $C^2$ set of $\R^n$. For any $q\ge2$ we set
$$\ov C(q)=\sup_{v\in W_0^{1,2}(E)\cap W^{2,q}(E)}\frac{\|D^2 v\|_q}{\|\Delta v\|_q}.$$
\end{definition}
We remark that $\ov C(q)$ is always finite and it is related to the Calder\'on-Zygmund Theorem. Moreover it is possible to show that there exists a constant $K$, not depending on $q$ (but depending on $E$), such that $\ov C(q)\le K q$. For the details see \cite{Yud}.
\par
Let us introduce our notion of solution, which retains more regularity than an ordinary weak solution. We want to focus the attention also on the set of test functions which is chosen in order to apply previous regularity results. In Remark \ref{testsmooth} we state the equivalence with other sets of test functions to which we will switch from time to time, as needed by the context. 

\begin{definition}\label{defnomu}
Let $\O$ be a bounded or exterior domain with boundary of class $C^2$ and $u_{\circ}\in L^2(\O)$. A field
$u\!:(0,T)\times \O\to\R^N$
 is said a solution of
system {\rm \eqref{PF}} if \be\label{ago151} u\in L^{p}(0,T; V(\O))\cap
C([0,T];L^2(\O))\,, \ t^\frac 1p \nabla u\in
L^{\infty}(0,T;L^p(\O))\,,\ee 
\be\label{ago252}u_t\in
L^{p'}(0,T;V(\O)')\,,\ t\,u_t\in L^{\infty}(0,T;L^2(\O))\,,\
t^\frac{p+2}{2p}\,\nabla u_t \in L^{2}(0,T; L^p(\O))\,,\ee
\be\label{testW12}\ba{ll}\dy \vs1\int_0^T [(u,\psi_t)-\left(|\nabla u|^{p-2}\,
\nabla u,\nabla \psi\right)]\,dt=
-(u_\circ, \psi(0)), \qquad \forall \psi\in
C^\infty_0([0,T)\times\O)
\ea\ee and
$$\lim_{t\to 0^+}\|u(t)-u_\circ\|_2=0\,.$$
\end{definition}

\begin{remark}\label{testsmooth}
We observe that, since $u(t)\in C([0,T];L^2(\O))$, by using a suitable cut-off function in time, we obtain that, for any $0\le s<t\le T$
\be\label{testsmoothst}\int_s^t [(u,\psi_\tau)-\left(|\nabla u|^{p-2}\,
\nabla u,\nabla \psi\right)]\,d\tau=
(u(t),\psi(t))
-(u(s), \psi(s)),\quad \forall \psi\in
C^\infty_0([0,T)\times\O).\ee
Moreover,
resorting to a density argument, we can take the test functions 
$\psi$ in the space $W^{1,2}(0,T;L^2(\O))\cap L^p(0,T;V(\O))$, obtaining an equivalent definition of solution which coincides with the one given in \cite{CM}. Always by density (see \cite[Prop. 23.23]{Zei}), $u$ is a solution in the sense of Definition \ref{defnomu}, if and only if, for any $0\le s<t\le T$
 \be\label{testW1p}\ba{ll}\dy \vs1\int_s^t [\langle u,\psi_\tau\rangle-\left(|\nabla u|^{p-2}\,
\nabla u,\nabla \psi\right)]\,d\tau=(u(t),
\psi(t))-(u(s), \psi(s)),\\ \hfill \forall \psi\in
V^{p,p'}(0,T;\O)\quad\mbox{and }u(0)=u_\circ.
\ea\ee
hence we can test the equation with the solution itself. 
%The converse is also true: if $u$ satisfies the above identity, then $u$ satisfies also \eqref{testW12}.
\end{remark}

In view of Sections \ref{L2} and \ref{Lq} we report, for the reader's convenience, three results on the regularity of the $p-$Laplacean elliptic system.
If we set
\be\label{rhat}\hat r=\left\{\ba{ll}\dy\vs1\frac{2n}{n(p-1)+2(2-p)}&\mbox{if }n\ge3\\
\dy\mbox{any number in }\left(2,\frac2{p-1}\right)&\mbox{if }n=2.\ea\right.\ee
we have
\begin{theorem}[{\cite[Theorem 1.2]{CGM}}]\label{CGMThm1.2}
Let $\O$ be a $C^2$ bounded or exterior domain of $\R^n$ and $p\in(1,2)$. Assume that $f\in L^{\hat r}(\O)\cap(\widehat W^{1,p}(\O))'$. Then the unique weak solution of the system
\be\label{ellipticPB}\ba{rl}-\nabla\cdot(|\nabla u|^{p-2}\nabla u)=f&\mbox{ in }\O\\
u=0&\mbox{ on }\pa\O\ea\ee
has second derivatives in $L^2(\O)$ and
$$\|D^2 u\|_2\le c\left(\|f\|_{-1,p'}^{\frac1{p-1}}+\|f\|_{\hat r}^{\frac1{p-1}}\right).$$
\end{theorem}

\begin{theorem}[{\cite[Theorem 1.1]{CM2014}}]\label{CMThm1.1}
Let be $E$ a bounded $C^2$ domain of $\R^n,\ n\ge2$, $p\in\left(2-\frac1{\ov C(2)},2\right)$ with $\ov C(2)$ as in Definition \ref{C(q)}. If $f\in L^q(E)$ with $q\ge\frac{2n}{n(p-1)+2(2-p)}$ for $n\ge3$ or $q>2$ for $n=2$ and
$$\hat q=\left\{\ba{ll}\frac{nq(p-1)}{n-q(2-p)}&\mbox{ if }q<n\\
\mbox{any number }<n&\mbox{ if }q=n\\
q&\mbox{ if }q>n\ea\right.$$
then there exists a unique $u$ solution of \eqref{ellipticPB} such that $u\in W_0^{1,\hat q}(E)\cap W^{2,\hat q}(E)$ and
$$\|u\|_{2,\hat q}\le c\|f\|_q^{\frac1{p-1}}.$$
\end{theorem}

\begin{theorem}[{\cite[Theorem 1.1]{CGM}}]\label{CGMThm1.1}
Let $\O$ be a $C^2$ exterior domain of $\R^n,\ n\ge2$. Assume that $f\in L^r(\O)\cap(\widehat W_0^{1,p}(\O))'$, with $r\in(n,\infty)$. Then, there exists $\ov p(r)\in(1,2)$ such that if $p\in(\ov p(r),2)$ there exists a unique solution $u$ of \eqref{ellipticPB} with
$$\|D^2 u\|_r\le c\left(\|f\|_{-1,p'}^{\frac1{p-1}}+\|f\|_r^{\frac1{p-1}}\right).$$
\end{theorem}

We observe that, by Remark \ref{testsmooth}, the notion of solution used in the above results (see \cite{CGM,CM2014}) is compatible with the one given in Definition \ref{defnomu}.

We end this section with a ``reverse'' version of the H\"older inequality (\cite[Theorem 2.12]{ADAFOU})

\begin{inequality}\label{reverse}Let $0<p<1$ and $p'=\frac p{p-1}$. If $f\in L^p(\O)$ and $0<\int_\O|g(x)|^{p'}\,dx<\infty$ then
$$\int_\O\left|f(x)g(x)\right|\,dx\ge\left(\int_\O\left|f(x)\right|^p\,dx\right)^{1/p}\left(\int_\O\left|g(x)\right|^{p'}\,dx\right)^{1/p'}.$$
\end{inequality}

\section{Existence results}\label{exi}

In the case of a bounded domain we quote here the following result taken from \cite[Theorem 1.1]{CM}. The statement is not exactly as the original one, where
the quantitative estimates are not present. They are somehow hidden in the proof and we want to make them explicit since we need them in view of the corresponding result in the case of an exterior domain. 

\begin{theorem}\label{existencebounded}
Let $E$ be a bounded $C^2$ subset of $\R^n$ and $u_0\in L^2(E)$. Then, for any $p\in(1,2)$, there exists a unique solution of problem \eqref{PF} in the sense of Definition \ref{defnomu}. Moreover we have the following estimates with constants $c$ not depending on $|E|$
\begin{align}\label{LinfinityL2}\vs1\|u(t)\|_2^2&\le2\|u_\circ\|_2^2&\mbox{for a.e. } t\in[0,T],\\
\label{nablauLinfinityLp}\vs1t^{\frac1p}\|\nabla u(t)\|_p&\le \|u_\circ\|_2^{\frac2p}&\mbox{for a.e. } t\in[0,T],\\
\label{LinftyW-1p'estimate}\vs1t^{\frac1{p'}}\|u_t(t)\|_{-1,p'}&\le c\|u_\circ\|_2^{\frac2{p'}}&\mbox{for a.e. }t\in[0,T],\\
\label{tu_tLinfinityL2}\vs1t\|u_t(t)\|_2&\le c\|u_\circ\|_2&\mbox{for a.e. }t\in[0,T],\\
\label{LpLp}\dy\vs1\int_0^T\|\nabla u(t)\|_p^p\,dt&\le \|u_\circ\|_2^2,\\
\label{tp+2nablaut}\dy\int_0^Tt^{\frac{p+2}2}\|\nabla u_t(t)\|_p^2\,dt&\le c\|u_\circ\|_2^{\frac4p}.\end{align}
\end{theorem}

\begin{proof}
The proof is based on a two steps approximation of the singular system via parabolic systems depending on two parameters. Furthermore the authors use the Faedo-Galerkin approximation method with smooth initial data and then they pass to the limit by density. It results that the estimates depend on four parameters and the passage to the limit has to be carefully managed. It is of no interest to replicate here the actual existence proof but, for the reader convenience, we perform only the formal computations treating the solution as it was smooth enough. We refer to the original paper \cite[Appendix]{CM} for the rigorous proof.\par
We begin with the classical energy estimate to get \eqref{LinfinityL2} and \eqref{LpLp}. We fix $s\in(0,T]$ and we multiply \eqref{PF}$_1$ by $u$. Integration in time and space gives
\be\label{energy}\frac12\|u(s)\|_2^2+\int_0^{s}\|\nabla u(t)\|_p^p\,dt\le \|u_\circ\|_2^2\ee
Now we multiply \eqref{PF}$_1$ by $u_t$ and integrate over $E$
$$\|u_t(t)\|_2^2+\frac 1p\frac d{dt}\|\nabla u(t)\|_p^p=0.$$
Multiplying by $t$ the above equation we get
$$t\|u_t\|_2^2+\frac1p\frac d{dt}\left(t\|\nabla u\|_p^p\right)=\frac1p\|\nabla u\|_p^p$$
and integrating this identity over $(0,s)$
\be\label{LinftyW1p}\int_0^{s}t\|u_t(t)\|_2^2\,dt+s\|\nabla u(s)\|_p^p\le\int_0^{s}\|\nabla u(t)\|_p^p\,dt.\ee
Hence by \eqref{LinftyW1p} and \eqref{energy}
\be\label{L2tu_t}\int_0^{s}t\|u_t(t)\|_2^2\,dt+s\|\nabla u(s)\|_p^p\le\|u_\circ\|_2^2\ee
which gives \eqref{nablauLinfinityLp}.
\par
Let us differentiate \eqref{PF} with respect to $t$ getting
$$u_{tt}-\nabla\cdot\left((p-2)|\nabla u|^{p-4}(\nabla u\otimes\nabla u)\cdot\nabla u_t+|\nabla u|^{p-2}\nabla u_t\right)=0.$$
Multiplying the above identity by $u_t$ and integrating over $E$ we obtain
$$\frac12\frac d{dt}\|u_t\|_2^2+\left\| |\nabla u|^{\frac{p-2}2}\nabla u_t\right\|_2^2\le(2-p)\left\| |\nabla u|^{\frac{p-2}2}\nabla u_t\right\|_2^2$$
and, multiplying by $t^2$
$$\frac12\frac d{dt}\left(t^2\|u_t\|_2^2\right)-t\|u_t\|_2^2+(p-1)t^2\left\| |\nabla u|^{\frac{p-2}2}\nabla u_t\right\|_2^2\le 0.$$
Finally, integrating in time over $(0,s)$, using \eqref{LinftyW1p} and \eqref{energy} we achieve
\be\label{gradugradut}\ba{l}\dy\vs1s^2\|u_s(s)\|_2^2+2(p-1)\int_0^s t^2\left\| |\nabla u|^{\frac{p-2}2}\nabla u_t\right\|_2^2\,dt\le 2\int_0^st\|u_t\|_2^2\,dt\\
\hfill\dy\le2\int_0^s\|\nabla u\|_p^p\,dt\le 2\|u_\circ\|_2^2\ea\ee
and \eqref{tu_tLinfinityL2} is proved.\par
By the definition of negative Sobolev norm and using estimate \eqref{L2tu_t} in \eqref{PF} we get
$$\|u_s(s)\|_{-1,p'}=\|\nabla\cdot\left(|\nabla u(s)|^{p-2}\nabla u(s)\right)\|_{-1,p'}\le\|\nabla u(s)\|_p^{p-1}\le  \frac c{s^{\frac{p-1}p}}\|u_\circ\|_2^{\frac{2(p-1)}p}$$
that gives \eqref{LinftyW-1p'estimate}.
\par
Concerning estimate \eqref{tp+2nablaut}, by H\"older's inequality with exponent $\frac2p,\ \frac2{2-p}$, using \eqref{L2tu_t} and \eqref{gradugradut}, we have
$$\ba{l}\dy\vs1\int_0^st^{\frac{p+2}p}\|\nabla u_t\|_p^2\,dt
=\int_0^s t^{\frac{2-p}p}\,t^2\left(\int_E|\nabla u|^{\frac{p(p-2)}2}|\nabla u_t|^p|\nabla u|^{\frac{p(2-p)}2}\,dx\right)^{\frac2p}\,dt\\
\hfill\dy\vs1\le\int_0^s t^2\left(\int_E|\nabla u|^{p-2}|\nabla u_t|^2\,dx\right)t^{\frac{2-p}p}\left(\int_E|\nabla u|^p\,dx\right)^{\frac{2-p}p}\,dt\\
\hfill\dy=\int_0^st^2\left\||\nabla u|^{\frac{p-2}2}\nabla u_t\right\|_2^2\left(t\|\nabla u\|_p^p\right)^{\frac{2-p}p}\,dt\le c\|u_\circ\|_2^2\|u_\circ\|_2^{\frac{2(2-p)}p}=c\|u_\circ\|_2^{\frac4p}.
\ea$$
\end{proof}

\begin{theorem}\label{existenceexterior}
The same results of Theorem \ref{existencebounded} hold true for an exterior $C^2$ domain.
\end{theorem}

\begin{proof}
To prove the thesis for an exterior domain we define a sequence of bounded sets invading $\O$. For any $k\in\N, \ k>R_0$, let be $u^k$ the unique solution of problem \eqref{PF} on $E_k:=\O\cap B(0,k)$ in place of $E$. 
First we extend $u^k$ to $0$ in $[0,T]\times(\O\setminus E_k)$ obtaining a function defined in $[0,T]\times\O$. We remark that the estimates \eqref{LinfinityL2}-\eqref{tp+2nablaut} in Theorem \ref{existencebounded} do not depend on the measure of the domain, hence we can consider all the norms computed on $\O$ instead of $E_k$.\par
Let be $k_0$ the smallest integer greater than $R_0$ and consider the sequence $\{u^k\}_{k\ge k_0}$.

By \eqref{LinfinityL2}, \eqref{tu_tLinfinityL2} and \eqref{LpLp} we can extract a subsequence (not relabeled) such that
\be\label{weakLinfinityL2}u^k\mathop{\rightharpoonup}^* u\qquad\mbox{weakly-}^* \mbox{ in }L^\infty(0,T;L^2(\O)),\ee
\be\label{weakLpV}u^k\rightharpoonup u\qquad\mbox{weakly in }L^p(0,T;V(\O)),\ee
\be\label{L2aet}u^k(t)\rightharpoonup u(t)\qquad\mbox{weakly in }L^2(\O)\ \mbox{for a.e. }t\in[0,T].\ee
%By Remark \ref{testsmooth} we can use smooth test functions to prove that $u$ is the solution of problem \eqref{PF} hence let us fix $t\in(0,T]$ and $\psi\in C^\infty_0\left([0,T]\times\O\right)$. Let $K$ be a bounded open set such that $\overline K\subset\O$ and $\psi(s,x)=0$ for all $(s,x)\in[0,T]\times(\R^n\setminus K)$.
\par
Let us fix $t\in(0,T]$ and $\psi\in V^{p,p'}(0,T;\O)$. 
By the weak convergences \eqref{weakLpV} and \eqref{L2aet} we get at once
\be\label{weaku_t}\int_0^t\left\langle u^k(s),\psi_s(s)\right\rangle\,ds\longrightarrow\int_0^t\left\langle u(s),\psi_s(s)\right\rangle\,ds,\ee
\be\label{weakL2ae}(u^k(t),\psi(t))\longrightarrow(u(t),\psi(t))\mbox{ for a.e. }t\in[0,T].\ee
By \eqref{LpLp} we get that there exists a function $\chi\in L^{p'}(0,T;L^{p'}(\O))$ such that
\be\label{weaktochi}|\nabla u^k|^{p-2}\nabla u^k\rightharpoonup \chi\quad\mbox{ weakly in }L^{p'}(0,T;L^{p'}(\O)).\ee
We want to prove that $\chi=|\nabla u|^{p-2}\nabla u$.
% $$\lim_n\int_0^T\left(|\nabla u^n|^{p-2}\nabla u^n,\nabla u^n\right)\,dt=\int_0^T(\chi,\nabla u)\,dt.$$
Since $u^k$ is a solution of problem \eqref{PF}, by Remark \ref{testsmooth}, we can use $u^k$ itself as a test function in \eqref{testW1p} getting
\be\label{nablaunnablaun}\ba{l}\dy\vs1\int_0^T\left(|\nabla u^k|^{p-2}\nabla u^k,\nabla u^k\right)\,dt=\int_0^T\langle u^k,u^k_t\rangle\,dt+(u_\circ,u^k(0))-(u^k(T),u^k(T))\\
\hfill\dy=-\frac12\|u^k(T)\|_2^2+\frac12\|u_\circ\|_2^2.\ea\ee
For any fixed $R>0$ let us consider the function $h_R$ defined in \eqref{hR}.
If $k>2R$, we can use $uh_R$ as a test function in equation \eqref{testW1p} to get
$$\int_0^T\left(|\nabla u^k|^{p-2}\nabla u^k,\nabla (uh_R)\right)\,dt=\int_0^T\langle u^k,(uh_R)_t\rangle\,dt+\|u_\circ(h_R)^\frac12\|_2^2-(u^k(T),u(T)h_R).$$
By \eqref{weaktochi}, \eqref{weaku_t} and \eqref{weakL2ae} we can pass to the limit as $k\to\infty$ in the above identity to gain
\be\label{chinablauhR}\int_0^T(\chi,\nabla (uh_R))\,dt=\int_0^T\langle u,(uh_R)_t\rangle\,dt+\|u_\circ(h_R)^\frac12\|_2^2-\|u(T)h_R^\frac12\|_2^2%=\frac12\|u_\circ\|_2^2-\frac12\|u(T)\|_2^2
.\ee
In order to pass to the limit as $R\to\infty$, we will examine each term separately.
$$\int_0^T(\chi,\nabla (uh_R))\,dt=\int_0^T(\chi,(\nabla u)h_R)\,dt+\int_0^T(\chi,u\otimes\nabla h_R))\,dt$$
As far as the first term is concerned, by dominated convergence we have
$$\int_0^T(\chi,(\nabla u)h_R)\,dt\longrightarrow\int_0^T(\chi,\nabla u)\,dt.$$
For the second one, 
considering that $\nabla h_R\not=0\iff R\le |x|\le 2R$, we have
$$|u\otimes\nabla h_R|\le c\left|\frac{u}R\right|\le2c\left|\frac ux\right|.$$
Since $1<p<2$, by Hardy inequality (it is not restrictive to suppose that $0\not\in\Omega$) we get
\be\label{hardy}\|u\otimes\nabla h_R\|_p\le c\left\|\frac{u}{x}\right\|_p\le c\|\nabla u\|_p<+\infty\ee
hence
$$\int_0^T\left|\left(\chi,u\otimes\nabla h_R\right)\right|\,dt\le\int_0^T\|\chi\|_{p'}\left\|u\otimes\nabla h_R\right\|_p\,dt\le c\int_0^T\|\chi\|_{p'}\|\nabla u\|_p\,dt<+\infty.$$
Once again we can apply the dominated convergence theorem to obtain
$$\int_0^T(\chi,u\otimes\nabla h_R)\,dt\longrightarrow 0.$$
Now we remark that 
\be\label{stronguhR}u h_R\longrightarrow u \quad\mbox{ strongly in }L^p(0,T;V(\Omega)).\ee
Indeed
$$\nabla u-\nabla(uh_R)=\nabla u(1-h_R)-u\otimes\nabla h_R$$
and, since $u\in L^p(0,T;V(\Omega))$, by the absolute continuity of the Lebesgue integral with respect to the domain of integration and \eqref{hardy}, we get the claim.
This allows us to pass to the limit in the term containing the time derivative
$$\int_0^T\langle u,(uh_R)_t\rangle\,dt=\int_0^T\langle uh_R,u_t\rangle\,dt\longrightarrow\int_0^T\langle u,u_t\rangle\,dt$$
by \eqref{stronguhR} and since $u_t\in L^{p'}(0,T;V(\Omega)')$.
In the end, by dominated convergence, we also get
$$\|u_\circ(h_R)^\frac12\|_2^2-\|u(T)h_R^\frac12\|_2^2\longrightarrow\|u_\circ\|_2^2-\|u(T)\|_2^2.$$
Collecting the above results and passing to the limit in \eqref{chinablauhR}, we have
\be\label{chinablau}\int_0^T(\chi,\nabla u)\,dt=\int_0^T\langle u,u_t\rangle\,dt+\|u_\circ\|_2^2-\|u(T)\|_2^2=\frac12\|u_\circ\|_2^2-\frac12\|u(T)\|_2^2
.\ee

By monotonicity we have
$$0\le\int_0^T\left(|\nabla u^k|^{p-2}\nabla u^k-|\nabla\psi|^{p-2}\nabla\psi,\nabla u^k-\nabla\psi\right)\,dt$$
Hence, using \eqref{nablaunnablaun}, \eqref{weaktochi}, \eqref{weakLpV} and the lower semicontinuity of the norm in the weak limit, we have
$$\ba{l}\dy\vs10\le\limsup_k-\frac12\|u^k(T)\|_2^2+\frac12\|u_\circ\|_2^2-\int_0^T(|\nabla u^k|^{p-2}\nabla u^k,\nabla\psi)\,dt\\
\hfill\dy\vs1-\int_0^T(|\nabla\psi|^{p-2}\nabla\psi,\nabla u^k-\nabla\psi)\,dt\\
\hfill\dy\le-\frac12\|u(T)\|_2^2+\frac12\|u_\circ\|_2^2-\int_0^T(\chi,\nabla\psi)\,dt-\int_0^T(|\nabla\psi|^{p-2}\nabla\psi,\nabla u-\nabla\psi)\,dt.
\ea$$
Substituting \eqref{chinablau} in the above inequality we get
$$0\le\int_0^T\left(\chi-|\nabla\psi|^{p-2}\nabla\psi,\nabla u-\nabla\psi\right)\,dt.$$
If we choose $\psi=u+\lambda\phi$ for generic $\phi\in V^{p,p'}(0,T;\O)$ and $\lambda\not=0$, we divide by $\lambda$ and finally we let $\lambda$ to 0, by the dominated convergence theorem we get that
$$\chi=|\nabla u|^{p-2}\nabla u.$$
Passing to the limit on $k$ in the definition of solution written for $u^k$, we get that $u$ is a solution in $\O$.
The estimates \eqref{nablauLinfinityLp}-\eqref{tp+2nablaut} for $u$ on $\O$ follow by the lower semicontinuity of the norms in the weak limits.
\end{proof}

\section{$L^q$ estimates for parabolic auxiliary problems}\label{adj}

In this section we deduce some estimates on the $L^q$ norm for the solution of some parabolic systems with smooth coefficients. The aim is to use them in the next section for the evaluation in $L^q$ of the time derivative of the solution of problem \eqref{PF}.
\par
Let $v(t,x)$ be a function such that
\be\label{beeta2}\sup_{s\in[0,t]}s\|\left(\mu+|\nabla v(s)|^2\right)^{\frac12}\|_p^p=:M(\mu,v)<+\infty.\ee
In order to apply known regularity results we introduce a time-space Friedrichs' mollifier $J_\eta$ and we define the following smooth tensor
\be\label{beeta}\ba{ll}\dy\vs1  (B_\eta(t;s,x))_{i\alpha j\beta}:=
\frac{\delta_{ij}\,\delta_{\alpha\beta}}{(\mu+|J_\eta(\nabla
v)(t-s,x)|^2)}{\null_{\frac {2-p}{2}}}\\
\dy\hskip5cm-(2-p)\frac{(J_\eta(\nabla v)\otimes J_\eta(\nabla
v))(t-s,x)}{(\mu+|J_\eta(\nabla v)(t-s,x)|^2)^{\frac
{4-p}{2}}}\,.\ea\ee
For any fixed $\sigma\in(0,t]$ and $\nu>0$ let us consider the parabolic problem
\be\ba{rll}\label{doubleadjoint}
\psi_\tau(\tau)-\nu\Delta\psi(\tau)-\nabla\cdot\left(B_\eta(t;\sigma-\tau,x)\nabla\psi(\tau)\right)&=0,&\mbox{ in }(0,\sigma)\times E,\\
\psi(\tau,x)&=0,&\mbox{ on }(0,\sigma)\times\pa E,\\
\psi(0,x)&=\psi_\circ(x),&\mbox{ on }\{0\}\times E.\ea\ee

\begin{lemma}\label{adjointLq}
Let $E$ be a bounded $C^2$ domain of $\R^n$. For any $\psi_\circ\in C^\infty_0(E)$ let $\psi$ be the unique solution of \eqref{doubleadjoint}. Then, for any $p\in(1,2)$ and $q\in\left[2,2+\frac{4(p-1)}{(2-p)^2}\right]$ it results
$$\|\psi(\tau)\|_q\le\|\psi_\circ\|_q,\qquad\forall\tau\in[0,\sigma].$$
\end{lemma}

\begin{proof}
The existence and uniqueness of the solution of \eqref{doubleadjoint} follows, for instance, by \cite[Theorem IV.9.1]{LSU} which also gives $\psi\in L^q(0,\sigma;W^{2,q}(E)\cap W^{1,2}_0(E)),\ \psi_\tau\in L^q(0,\sigma;L^q(E))$. For brevity of notation we set
$$a_\eta(\mu,v):=\left(\mu+|J_\eta(\nabla
v)|^2\right)^\frac{(p-2)}{2}\,.$$

Since $q\ge2$ we can multiply the system by $|\psi|^{q-2}\psi$ and integrate over $E$ obtaining
$$\ba{l}\dy\vs1\frac1q\frac d{dt}\|\psi\|_q^q+\nu\int_E|\psi|^{q-2}|\nabla\psi|^2\,dx+\nu(q-2)\int_E|\psi|^{q-4}|\nabla\psi\cdot\psi|^2\,dx\\
\hfill\dy\vs1+\int_E a_\eta(\mu,v(\sigma-\tau))|\nabla\psi|^2|\psi|^{q-2}\,dx
+(q-2)\int_E a_\eta(\mu,v(\sigma-\tau))|\psi|^{q-4}|\nabla\psi\cdot\psi|^2\,dx\\
\hfill\dy\vs1=(p-2)\int_E\frac{|\psi|^{q-2}|J_\eta(\nabla v(\sigma-\tau))\cdot\nabla\psi|^2}
{\left(\mu+|J_\eta(\nabla v(\sigma-\tau))|^2\right)^{\frac{4-p}2}}\,dx\\
\hfill\dy\vs1+(p-2)(q-2)\int_E\frac{|\psi|^{q-4}(J_\eta(\nabla v(\sigma-\tau)\cdot\nabla\psi)(J_\eta(\nabla v(\sigma-\tau)\cdot\psi)(\nabla\psi\cdot\psi)}
{\left(\mu+|J_\eta(\nabla v(\sigma-\tau))|^2\right)^{\frac{4-p}2}}\,dx\\
\hfill\dy=:(p-2)I_1+(p-2)(q-2)I_2.\ea
$$
We observe that
$$|I_1|\le\int_E a_\eta(\mu,v(\sigma-\tau))|\psi|^{q-2}|\nabla\psi|^2\,dx=:J_1.$$
By Cauchy-Schwarz's and H\"older's inequalities it results
$$\ba{l}\dy\vs1|I_2|\le\int_E\left(\mu+|J_\eta(\nabla v(\sigma-\tau))|^2\right)^{\frac{p-4}2}|\psi|^{q-4}|J_\eta(\nabla v(\sigma-\tau))|^2|\nabla\psi||\psi||\nabla\psi\cdot\psi|\,dx\\
\hfill\dy\vs1\le
\left(\int_E a_\eta(\mu,v(\sigma-\tau))|\psi|^{q-4}|\nabla\psi|^2|\psi|^2\,dx\right)^{\frac12}
\left(\int_E a_\eta(\mu,v(\sigma-\tau))|\psi|^{q-4}|\nabla\psi\cdot\psi|^2\,dx\right)^{\frac12}\\
\hfill\dy=:J_1^{\frac12}J_2^{\frac12}\ea$$
hence
\be\label{Lqdecresing}\frac1q\frac d{dt}\|\psi\|_q^q+(p-1)J_1+(q-2)J_2\le(2-p)(q-2)J_1^{\frac12}J_2^{\frac12}\le\frac1{2\ve}J_1+\frac\ve2(2-p)^2(q-2)^2J_2.\ee
We want to choose $\ve$ such that
\be\label{rangeforeps}\frac1{2\ve}\le p-1,\qquad\frac\ve2 (2-p)^2(q-2)^2\le q-2\ee
and this is always possible if 
$$\frac1{2(p-1)}\le\frac2{(q-2)(2-p)^2}.$$
An easy computation shows that the above inequality is verified for any $p\in(1,2)$ if
$$q\in\left[2,2+\frac{4(p-1)}{(2-p)^2}\right].$$
Choosing an $\ve$ satisfying \eqref{rangeforeps} and substituting it in \eqref{Lqdecresing} we get
$$\frac1q\frac d{dt}\|\psi\|_q^q\le0.$$
\end{proof}

For any fixed $t\in(0,T), \nu>0$ and $\vp_\circ\in C^\infty_0(E)$ we consider the following problem, adjoint of \eqref{doubleadjoint}

\be\label{AD1}\begin{array}{rll}\dy\vs1 \vp_s(s)-\nu\Delta \vp(s)- \nabla\cdot
(B_\eta(t;s,x)\nabla \vp(s))&= 0\,,&\hskip-0.2cm\textrm{ in }(0,t)\times E,\\
\dy\vs1\hskip3.5cm \vp(s,x)&=0\,,&\hskip-0.2cm\textrm{ on }
(0,t)\times\pa E,\;\\
\dy\hskip3.5cm \vp(0,x)&=\vp_\circ(x),&\hskip-0.2cm\mbox{ on
}\{0\}\times E,\end{array}\ee

\begin{lemma}
Let $E$ be a bounded $C^2$ domain of $\R^n$. For any $\vp_\circ\in C^{\infty}_0(E)$ let $\vp$ be the unique solution of \eqref{AD1}. Then for any $p\in(1,2)$ and $r\in\left[2-\frac{4(p-1)}{p^2},2\right]$
$$\|\vp(s)\|_r\le\|\vp_\circ\|_r,\qquad\forall s\in[0,t].$$
\end{lemma}

\begin{proof}
For any arbitrary function $\psi_\circ\in C^\infty_0(E)$ and $\sigma\in[0,t]$,  let $\psi$ be the solution of problem \eqref{doubleadjoint}. 
The system \eqref{AD1} has an unique solution, by \cite[Theorem IV.9.1]{LSU}, and the solution $\vp$ is also regular enough to multiply \eqref{AD1} by $\psi(\sigma-s)$. Integrating the product by parts on $[0,\sigma]\times\O$ gives
$$\ba{l}\dy\vs1\left(\vp(\sigma),\psi_\circ\right)-\left(\vp_\circ,\psi(\sigma)\right)+\int_0^\sigma\left(\vp(s),\psi_s(\sigma-s)\right)\,ds
-\nu\int_0^\sigma\left(\vp(s),\Delta\psi(\sigma-s)\right)\,ds\\
\hfill\dy-\int_0^\sigma\left(\nabla\cdot\left(B_\eta(t;s)\nabla\psi(\sigma-s)\right),\vp(s)\right)\,ds=0.\ea$$
Since $\psi$ is a solution of \eqref{doubleadjoint}, substituting in the integrals $\sigma-s=\tau$, we get
\be\label{dualpsiphi}\left(\vp(\sigma),\psi_\circ\right)=\left(\vp_\circ,\psi(\sigma)\right).\ee
Since $r\in\left[\frac{2p^2+4}{p^2+4p},2\right]$ we have that $r'\in\left[2,2+\frac{4(p-1)}{(2-p)^2}\right]$ and we can apply Lemma \ref{adjointLq} to get
$$\left|\left(\vp_\circ,\psi(\sigma)\right)\right|\le\|\vp_\circ\|_r\|\psi(\sigma)\|_{r'}\le\|\vp_\circ\|_r\|\psi_\circ\|_{r'}$$
for any $\psi_\circ\in C^\infty_0(E)$.
By a density argument and \eqref{dualpsiphi} we get the thesis.
\end{proof}

\begin{lemma}\label{phiL2Lr}
Let $E$ be a bounded $C^2$ domain of $\R^n$, $p>\frac{2n}{n+2},\ r\in\left[2-\frac{4(p-1)}{p^2},2\right]$ and $\vp_\circ\in C^\infty_0(E)$. If $\vp$ is the solution of problem \eqref{AD1} we have
$$\|\vp(s)\|_2\le c M(\mu,v)^{\frac{(2-p)\gamma}2}\|\vp_\circ\|_r\left(t^{\frac1p}-(t-s)^{\frac1p}\right)^{-p\gamma},\qquad\forall s\in(0,t],$$
with $M(v,\eta)$ defined in \eqref{beeta2} and
\be\label{Maxpostraa} \gamma=\gamma(r):=\mbox{\large$
\frac{n(2-r)}{r(2p-2n+np)}$}\,.\ee
\end{lemma}

\begin{proof}
We refer to \cite[Lemma 2.4]{CM} remarking that even if the range for $r$ is different, the proof remains unchanged.
\end{proof}

\section{Estimates for the time derivative}\label{ut}

We begin the section gathering some results taken from \cite[Section 3]{CM} concerning the following non-singular ($\mu>0$, $\nu>0$) parabolic system on the bounded $C^2$ domain $E$
\be\label{PFepv}\begin{array}{rll}\dy\vs1 v_t-\nu\Delta
v-
\nabla\cdot\left(\left(\mu+|(\nabla
v)|^2\right)^\frac{(p-2)}{2}\nabla
v\right)&\hskip-.3cm= 0\,,&\hskip-0.2cm\textrm{ in
}(0,T)\times E,
\\\dy\vs1 v(t,x)&\hskip-.3cm=0\,,&\hskip-0.2cm \textrm{ on }
(0,T)\times\partial E,\;\\\dy v(0,x)&\hskip-.3cm=v_\circ(x),&\hskip-0.2cm\mbox{ on
}\{0\}\times E\,.\end{array}\ee %

We have the following results, for which we refer to \cite[Propositions 3.1 and 3.2]{CM}

\begin{proposition}\label{existence}
{\sl Let $\nu>0$, $\mu>0$ and $p\in(1,2)$. Assume that $v_\circ$
belongs to $C_0^\infty(E)$. Then there exists a unique weak solution $v$
of system \eqref{PFepv} such that
$$ v\in C(0,T;L^2(E))\cap L^{2}(0,T;
W_0^{1,2}(E)),$$
$$ v_t\in L^{\infty}(0,T;L^2(E))\cap
L^2(0,T;W^{1,2}(E)),$$
$$\lim_{t\to
0^+}\|v(t)-v_\circ\|_2=0\,.$$
Moreover
%\item[i)]
%$\dy \|v\|_{L^{\infty}(0,T;L^2(\O))}+\|v\|_{L^{p}(0,T; V)}\leq \|v_\circ\|_2+ c\, B(\mu,v_\circ)^\frac{1}{p}\,,$
% \item[ii)] $\|v_t\|_{L^{\infty}(0,T;L^2(\O))}+\|\nabla v_t\|_{L^{2}(0,T; L^p(\O))}$\\
% \null\hfill$\dy \leq\,c\,(\nu+\mu^\frac{p-2}{2}) \|v_\circ\|_{2,2}
% \left( \nu\,\|\nabla v_\circ\|_2^2+ \|\nabla v_\circ\|_2^p+\mu^\frac
% p2T|\O|\right)^\frac{2-p}{2p};$
%\item[iii)]  $\|v\|_{L^2(0,T; W_0^{1,2}(\O))}
%+\|\nabla v_t\|_{L^2(0,T;
%L^2(\O))}$\\ 
%$\dy \hspace*{27ex}\leq\,\frac{c}{\sqrt \nu} \,\|v_\circ\|_2+\left(\frac{c}{\sqrt \nu} \,\|v_\circ\|_2\right)^{p-1}+
%\frac{c}{\sqrt \nu}\,(\nu+\mu^\frac{p-2}{2}) \|v_\circ\|_{2,2}
%$,
\be\label{iv}\vspace{0.5ex}\|t^\frac 1p \nabla v\|_{L^{\infty}(0,T;L^p(E))}\leq c\,B(\mu,v_\circ)^\frac 1p,\ee
\be\label{v}\|t\,v_t\|_{L^{\infty}(0,T;L^2(E))}\le c B(\mu,v_\circ)^{\frac12}\ee
%\item[v)]  $\dy \|t\,\nabla
%v_t\|_{L^2(0,T; L^2(\O))}\leq \frac{c}{\sqrt \nu}\, B(\mu,v_\circ)^\frac 12$,
where
\be\label{defB}B(\mu, w):= \,\|w\|_2^2+\mu^\frac p2T|E|\,. \ee
}
\end{proposition}

With this tool at our disposal we can state the following crucial result

\begin{proposition}\label{ut1} {\sl Let $p>\frac{2n}{n+2}$, $q\in  \left[2,2+\frac{4(p-1)}{(2-p)^2}\right]$ and $\O$ a bounded or exterior $C^2$ domain of $\R^n$.
 Let $u$ be the unique solution of \eqref{PF} corresponding to
$u_\circ\in L^2(\O)$. Then $t^{1+\gamma}\,u_t\in
L^\infty(0,T;L^{q}(\O))$,
 with  $\gamma=\gamma(q')$ given by \eqref{Maxpostraa}. Moreover the following estimate holds
 \be\label{estvt}\|u_t(t)\|_{q}\leq \frac{c}{t}{\null_{1+\gamma}}\,
\|u_\circ\|_2^{(2-p)\gamma+1},\, \mbox{ a.e in }(0,T).\ee}
\end{proposition}

\begin{proof}
The proof follows substantially the one of \cite[Proposition 5.1]{CM}. For the reader's convenience we reproduce here only the main lines to make clear the fundamental role played by the adjoint problem \eqref{AD1}.  First we consider a bounded $C^2$ domain $E$ and a solution $v$ of the system \eqref{PFepv}. We have to keep in mind that $v$ depends on the parameters $\nu,\mu$ an also another one, say $m$, used in the approximation of the initial data in $L^2(E)$ by means of smooth functions. We regularize \eqref{PFepv}$_1$ in time, introducing another parameter $\rho$ arising from the mollifier, and we differentiate with respect to $t$. Finally we multiply the result by $\vp(t-\tau)$ where $\vp$ is a solution of \eqref{AD1} hence it depends on $\eta$. Omitting the indexes $\nu,\mu,m,\rho,\eta$ and performing only the formal computations we get
$$\ba{l}\dy\vs1\left(v_{\tau\tau}(\tau)-\nu\Delta v_\tau(\tau)-(p-2)\nabla\cdot\left(\mu+|\nabla v(\tau)|^2\right)^{\frac{p-4}2}\left(\nabla v(\tau)\otimes\nabla v(\tau)\right)\cdot\nabla v_\tau(\tau)\right.\\
\hfill\dy\left.+\left(\mu+|\nabla v(\tau)|^2\right)^{\frac{p-2}2}\nabla v_\tau(\tau)\right)\cdot\vp(t-\tau)=0\ea$$
An integration of the above identity on $E$ and between $\frac t2$ and $t$ with respect to $\tau$, provides
$$\ba{l}\dy\vs1\left(v_t(t),\vp_\circ\right)-\left(v_t\left(\frac t2\right),\vp\left(\frac t2\right)\right)\\
\hfill\dy\vs1=-\int_{\frac t2}^t\left(v_\tau(\tau),\vp_\tau(t-\tau)\right)\,d\tau-\nu\int_{\frac t2}^t\left(\nabla v_\tau(\tau),\nabla\vp(t-\tau)\right)\,d\tau\\
\hfill\dy\vs1-(p-2)\int_{\frac t2}^t\left(\left(\mu+|\nabla v(\tau)|^2\right)^{\frac{p-4}2}\left(\nabla v(\tau)\otimes\nabla v(\tau)\right)\cdot\nabla v_\tau(\tau),\nabla\vp(t-\tau)\right)\,d\tau\\
\hfill\dy\vs1-\int_{\frac t2}^t\left(\left(\mu+|\nabla v(\tau)|^2\right)^{\frac{p-2}2}\nabla v_\tau(\tau),\nabla\vp(t-\tau)\right)\,d\tau\\
\hfill\dy=\int_{\frac t2}^t\left(v_\tau(\tau),-\vp_\tau(t-\tau)+\nu\Delta\vp(t-\tau)\right)\,d\tau\\
\hfill\dy\vs1-(p-2)\int_{\frac t2}^t\left(\frac{\nabla v(\tau)\otimes\nabla v(\tau)\cdot\nabla\vp(t-\tau)}{\left(\mu+|\nabla v(\tau)|^2\right)^{\frac{4-p}2}},\nabla v_\tau(\tau)\right)\,d\tau\\
\hfill\dy\vs1-\int_{\frac t2}^t\left(\frac{\nabla\vp(t-\tau)}{\left(\mu+|\nabla v(\tau)|^2\right)^{\frac{2-p}2}},\nabla v_\tau(\tau)\right)\,d\tau\\
\hfill\dy\vs1=\int_{\frac t2}^t\left(v_\tau(\tau),-\vp_\tau(t-\tau)+\nu\Delta\vp(t-\tau)\right)\,d\tau\\
\hfill\dy+\int_{\frac t2}^t\left(\nabla\cdot\left(\frac{\nabla v(\tau)\otimes\nabla v(\tau)\cdot\nabla\vp(t-\tau)}{\left(\mu+|\nabla v(\tau)|^2\right)^{\frac{4-p}2}}+\frac{\nabla\vp(t-\tau)}{\left(\mu+|\nabla v(\tau)|^2\right)^{\frac{2-p}2}}\right),v_\tau(\tau)\right)\,d\tau.\ea$$
At this point we have to remark that if we replace $\nabla v$ with $J_\eta(\nabla v)$ in the denominators of the last integral, we obtain that the right-hand side is zero, since $\vp$ is a solution of \eqref{AD1}. This can be made rigorous by a careful passage to the limit as $\eta$ goes to 0. The details are completely described in the proof of \cite[Proposition 5.1]{CM}. In the end, by \eqref{v} and Lemma \ref{phiL2Lr}, we get
$$\ba{l}\dy\vs1(v_t(t),\vp_\circ)=\left(v_t\left(\frac t2\right),\vp\left(\frac t2\right)\right)\le \left\|v_t\left(\frac t2\right)\right\|_2\left\|\vp\left(\frac t2\right)\right\|_2\\
\hfill\dy\le \frac{c B(\mu,v_\circ)^{\frac12}}t\,\frac{M(\mu,v)^{\frac{(2-p)\gamma}2}\|\vp_\circ\|_{q'}}{t^{\gamma}}.\ea
$$
Using the definition of $M(\mu,v)$ given in \eqref{beeta2} and \eqref{iv}, we get
$$M(\mu,v)\le c\left(\mu^{\frac p2}T|E|+B(\mu,v_\circ)\right)\le c B(\mu,v_\circ)$$
hence
$$(v_t(t),\vp_\circ)\le \frac{c B(\mu,v_\circ)^{\frac{(2-p)\gamma}2}}{t^{1+\gamma}}\|\vp_\circ\|_{q'}$$
for any $\vp_\circ\in L^{q'}(E)$.
It follows that $v_t(t)\in L^q(E)$ and $\|v_t(t)\|_q\le\frac{c B(\mu,v_\circ)^{\frac{(2-p)\gamma}2}}{t^{1+\gamma}}$.
To conclude the proof we need to pass to the limit in all the parameters. The process is quite involved and it is described in \cite[Proposition 3.2 and Theorem 1.1]{CM}. The result is the convergence of $v$ to the solution $u$ of \eqref{PF} likewise the smooth initial data $v_\circ$ approximate $u_\circ$ in $L^2(E)$.
Moreover we get that $B(\mu,v_\circ)\to c\|u_\circ\|_2^2$ and the thesis for a bounded domain follows.
\par
To extend the result to an exterior domain $\O$ we use the same sequence $\{E_k\}$ of bounded sets invading $\O$ as in the proof of Theorem \ref{existenceexterior}. In estimate \eqref{estvt} the norm of $u_\circ$ is evaluated on $E_k$ but it can be increased uniformly with respect to $k$ to the norm on the whole $\O$. Hence we have that $t^{1+\gamma}u(t)\in L^\infty(0,T;L^q(\O))$ and \eqref{estvt} holds true also in $\O$.
\end{proof}

\section{$L^2$ estimates for $D^2u$.}\label{L2}

In this section we prove $L^\infty(\ve,T;L^2(\O))$ estimates for the solution of problem \eqref{PF}. Despite the fact that our main interest goes towards exterior domains, we consider also the case of a bounded domain. Indeed, in this case, we improve \cite[Theorem 1.2]{CM} removing some constraints on $p$ and moving down its lower bound. 

\begin{proof}[Proof of Theorem \ref{W22IN}]
Let us fix $t>0$ and consider the system \eqref{PF} as an elliptic problem in the variable $x$. By Proposition \ref{ut1} we get that $u_t(t)\in L^q(\O)$ for any $q\in\left[2,2+\frac{4(p-1)}{(2-p)^2}\right]$. We want to apply Theorem \ref{CGMThm1.2} using $u_t(t)$ as the force term. To this aim, we show that $u_t$ belongs to $L^{\hat r}(\O)$ with $\hat r$ defined in \eqref{rhat}.
We remark that the number $2+\frac{4(p-1)}{(2-p)^2}$ is an increasing quantity with respect to $p\in(1,2)$. In our hypotheses $p\in\left(\frac{2n}{n+2},2\right)$, hence we have that, for any $p$ in this interval,
$$2+\frac{4(p-1)}{(2-p)^2}>\frac{n^2+4}4.$$
We need to compare the two quantities $\frac{n^2+4}4$ and $\hat r$. Consider first the case $n\ge3$.
By a straightforward computation it is easy to check that
$$\frac{n^2+4}4>\frac{2n}{n(p-1)+2(2-p)}\iff p>\frac{n^3-4n^2+12n-16}{(n^2+4)(n-2)}$$
and
$$\frac{n^3-4n^2+12n-16}{(n^2+4)(n-2)}<\frac{2n}{n+2}\iff n^4-2n^3+4n^2-24n+32>0.$$
But
$$n^4-2n^3+4n^2-24n+32=(n-2)^2(n^2+2n+8)>0\qquad \forall n\in\N, n\ge3.$$
Hence, for any $p\in\left(\frac{2n}{n+2},2\right)$ we have that $\hat r\in\left[2,2+\frac{4(p-1)}{(2-p)^2}\right]$.
\par
If $n=2$ it is enough to observe that the intersection $\left(2,\frac2{p-1}\right)\cap\left[2,2+\frac{4(p-1)}{(2-p)^2}\right]$ is not empty.
\par
In both cases,
by Proposition \ref{ut1}, $u_t(t)\in L^{\hat r}(\O)$ and, by Theorem \ref{existenceexterior}, $u_t(t)\in W^{-1,p'}(\O)$.  
We can apply Theorem \ref{CGMThm1.2} with $f(x)=u_t(t,x)$ obtaining that $u(t)\in W^{2,2}(\O)$ and, by \eqref{LinftyW-1p'estimate}, \eqref{estvt} 
\be\label{D2L2E}\|D^2 u(t)\|_2\le c\left(\|u_t(t)\|_{\hat r}^{\frac1{p-1}}+\|u_t(t)\|_{-1,p'}^{\frac1{p-1}}\right)\le \frac c{t^{\frac{1+\ov\gamma}{p-1}}}\|u_\circ\|_2^{\frac{(2-p)\ov\gamma+1}{p-1}}+\frac c{t^{\frac1p}}\|u_\circ\|_2^{\frac2p}\ee
with 
$\ov\gamma=\gamma(\hat r')=\frac{(n-2)(2-p)}{p(n+2)-2n}$ if $n\ge3$ (for the notation see \eqref{Maxpostraa} and \eqref{rhat}) or $\ov\gamma=\gamma(r')=\frac{r-2}{r(p-2)}$ for any $r\in\left(2,\frac2{p-1}\right)\cap\left[2,2+\frac{4(p-1)}{(2-p)^2}\right]$ if $n=2$. 

\end{proof}

\section{Higher integrability of $D^2u$}\label{Lq}

In this section we increase the integrability of $D^2u$ to a power greater than 2. The greatest exponent of integrability depends on $p$ and increases as $p$ approaches 2 from below. In a fashion that is common to this kind of results, see \cite{BDV,CBDV,CGM,CM,CM2014}, the range for $p$ is constrained to be close to 2 in dependence of the summability $q$ required for the second derivatives. For a bounded domain, the following theorem improves the previous result obtained in \cite[Theorem 1.2]{CM} extending the range for $q$ and removing some constraints on $p$.

\begin{proof}[Proof of Theorem \ref{W2qbounded}]
We set $g(p)=2+\frac{4(p-1)}{(2-p)^2}$ and we remark that $g$ is an increasing function on the interval $(1,2)$. Hence $$g(p)\ge g\left(\frac{2n}{n+2}\right)=1+\frac{n^2}4\qquad\forall p\in\left(\frac{2n}{n+2},2\right).$$
Since $1+\frac{n^2}4>n$ for any $n\ge3$ we have that $g(p)>n$ for any $p,n$ in the hypotheses of our theorem, hence the interval for $q$ goes beyond $n$ (in the case $n=2$ the whole interval is trivially beyond $n$). Since the behavior is different for $q$ over, behind or equal to $n$, we will distinguish three cases.
\par
Let us consider first the case $q\in\left(n,2+\frac{4(p-1)}{(2-p)^2}\right]$. By Proposition \ref{ut1} we have that $u_t\in L^\infty(\ve,T;L^q(E))$. Since $p>\max\{\phi(2),\phi(q)\}$ and $q>n$ we can use $u_t$ as the force term in Theorem \ref{CMThm1.1} to get that $u\in L^\infty(\ve,T;W^{2,q}(E))$.
\par
If $q=n$, choose any $p>\inf\limits_{q>n}\left\{2-\frac1{\ov C(q)}\right\}$. There exists $q_1>n$ such that $p>2-\frac1{\ov C(q_1)}$ and, by Proposition \ref{ut1}, $u_t\in L^\infty(\ve,T;L^{q_1}(E))$. Again, by Theorem \ref{CMThm1.1}, $u\in L^\infty(\ve,T;W^{2,q_1}(E))\subset L^\infty(\ve,T;W^{2,n}(E))$.
\par
If $q\in[2,n),\ n\ge3$ then let $\ov q=\frac{qn}{n(p-1)+q(2-p)}$. We remark that, since $2\le q<n$, we have
$$\ov q> \frac{2n}{n(p-1)+n(2-p)}=2,\quad \ov q<\frac{qn}{q(p-1)+q(2-p)}=n<g(p).$$
Applying Proposition \ref{ut1} we get that $u_t\in L^\infty(\ve,T;L^{\ov q}(E))$. Following the notation of Theorem \ref{CMThm1.1} we have that $\hat{\ov q}=q$ and $\ov q>\frac{2n}{n(p-1)+2(2-p)}$. Since $p>\phi(q)$ we can apply the quoted theorem achieving that $u\in L^\infty(\ve,T;W^{2,q}(E)).$ 
\end{proof}

\begin{proof}[Proof of Theorem \ref{W2qexterior}]
Let us fix $q\in\left[2,2+\frac{4(p-1)}{(2-p)^2}\right]$ and $\ve>0$. By Theorem \ref{W22IN} we have that $u_t\in L^\infty(\ve,T;L^q(\O))$ and by Theorem \ref{existenceexterior} $u_t\in L^\infty(\ve,T;(\widehat W_0^{1,p}(\O))')$. As in the proof of Theorem \ref{W2qbounded} we consider different ranges for $q$.\par
If $q>n$ then, by Theorem \ref{CGMThm1.1} there exists $\ov p(q)$ such that for any $p\in(\ov p(q),2)$, $D^2u\in L^\infty(\ve,T;L^q(\O))$.
\par
If $2\le q\le n$ we set $\hat p=\inf\left\{\ov p(q):n<q\le2+\frac{4(p-1)}{(2-p)^2}\right\}$. For any $p>\hat p$ there exists $\ov q_1>n$ such that $p>\ov p(q_1)$. Using the result just achieved we get that $D^2u\in L^\infty(\ve,T;L^{q_1}(\O))$. By Theorem \ref{W22IN} $D^2u\in L^\infty(\ve,T;L^2(\O))$ and we get the thesis by interpolation.
\end{proof}

The above result is at first sight a little bit confusing about a sort of cross reference between $p$ and $q$. Which one depends on the other, or there is simply a mutual dependence between them? For instance, if we choose $p>\frac{2n}{n+2}$, which is the best $q$ we can reach? We point out that, generally speaking, for an exterior domain, the highest does not necessarily means the best. Fortunately, for our solution we always have $D^2u\in L^\infty(\ve,T;L^2(\O))$ hence we can interpolate and the best $q$ is actually the highest we can achieve. Hence, for fixed $p>\frac{2n}{n+2}$ we can expect to find $q$ which is at best $2+\frac{4(p-1)}{(2-p)^2}$. But now we have to take one step back and check if $p>\ov p(q)$. In this framework, the best $q$ is $\sup\{q\le2+\frac{4(p-1)}{(2-p)^2}:\ov p(q)<p\}$. As a result, the statement of the theorem is not very charming, especially because we are not able to prove (although it sounds very reasonable) that the quantity $\ov p(q)$ is increasing with respect to $q$. Hence we choose to give the result leaving someway implicit the relation between $p$ and $q$. On the other side, if we ask which is the lowest $p$ that is allowed to get $D^2 u$ in $L^q$ there is a very clean answer which is stated in the following

\begin{corollary}\label{W2qcorollary}
Let $\O$ be an exterior $C^2$ domain of $\R^n$ and $q>2$. If 
$$p>\max\left\{\frac{2n}{n+2},\frac{2(q-1-\sqrt{q-1})}{q-2},\ov p(q)\right\}$$
with $\ov p(q)$ as in Theorem \ref{W2qexterior}, then $D^2u\in L^\infty(\ve,T;L^q(\O))$, where $u$ is the unique solution of \eqref{PF}.
\end{corollary}

\begin{proof}
The result follows straightforward solving the inequality $q\le2+\frac{4(p-1)}{(2-p)^2}$ with respect to $p$. We only remark that the condition $q>2$ in the statement is not necessary, since for $q=2$ the quantity $\frac{2(q-1-\sqrt{q-1})}{q-2}$ tends to 1, giving a trivial constraint for $p$. This is perfectly in line with the result of Theorem \ref{W2qexterior} which holds true also for $q=2$. Nevertheless we want to notice that the case $q=2$ is covered also by Theorem \ref{W22IN} which is sharper, since allows the whole range $p\in\left(\frac{2n}{n+2},2\right)$. 
\end{proof}

\section{H\"older continuity of $\nabla u$}\label{C1alpha}

In this section we investigate the H\"older continuity of $\nabla u$, up to the boundary of the exterior domain $\O$. We start introducing the relevant quantity for evolution problems that is the parabolic H\"older seminorm defined by
\be\label{holderseminorm}[w]_\lambda=\sup_t\sup_{x\not=y}\frac{|w(t,x)-w(t,y)|}{|x-y|^\lambda}+\sup_x\sup_{t\not=s}\frac{|w(t,x)-w(s,x)|}{|t-s|^{\frac\lambda2}}.\ee
We rely on the following result on Bochner spaces (see \cite[Theorem 2.1]{Sol77} and \cite[Lemma 2.7]{CM}).

\begin{lemma}\label{solonnikov}
Let $\O$ be a bounded or exterior $C^2$ domain of $\R^n$, $\ve>0$ and $q>n$. There exists a constant $C$, such that if $u\in L^\infty(\ve,T;W^{2,q}(\O)\cap W_0^{1,q}(\O))$ and $u_t\in L^\infty(\ve,T;L^q(\O))$ then
$$[\nabla u]_\lambda\le C\left(\sup_t\left(\|u_t(t)\|_q+\|D^2u(t)\|_q\right)+\sup_t\|u(t)\|_q\right)$$
with $\lambda=1-\frac nq$.
\end{lemma}

\begin{proof}
The case $\O$ bounded is considered in \cite[Lemma 2.7]{CM}. If $\O$ is exterior it is enough to remark that \cite[Theorem 2.1]{Sol77} is based on an extension argument and does not make use of the boundedness of $\O$.
\end{proof}

Let $u$ be the solution of problem \eqref{PF}. We choose $q>n$ and we consider $p$ satisfying the hypotheses of Corollary \ref{W2qcorollary}. By the Sobolev-Nirenberg-Gagliardo inequality we have
\be\label{interpolation}\|u(t)\|_q\le c\|D^2u(t)\|_q^\theta\|u(t)\|_2^{1-\theta}\ee
with $\theta=\frac{(q-2)n}{(q-2)n+4q}$. Hence, by Corollary \ref{W2qcorollary} and \eqref{LinfinityL2}, $u\in L^\infty(\ve,T;W^{2,q}(\O))$. Since $p>\frac{2(q-1-\sqrt{q-1})}{q-2}$ we have that $q<2+\frac{4(p-1)}{(2-p)^2}$ hence, by Proposition \ref{ut1}, $u_t\in L^\infty(\ve,T;L^q(\O))$. We are now in the position to apply Lemma \ref{solonnikov} and obtain the H\"older continuity of $\nabla u$. Gathering together the estimates for $\|u_t\|_q,\ \|D^2 u\|_q$ and the interpolation estimate \eqref{interpolation} we can formulate the following result. In a fashion similar to Corollary \ref{W2qcorollary}, we write the statement choosing the H\"older exponent $\lambda$ and finding the correct range for $p$

\begin{theorem}
Let $\O$ be an exterior $C^2$ domain of $\R^n$, $\ve>0$ and $\lambda\in(0,1)$. If 
$$p>\max\left\{\frac{2n}{n+2},\frac{2(n+\lambda-1-\sqrt{(1-\lambda)(n+\lambda-1)}}{n+2\lambda-2},\ov p\left(\frac n{1-\lambda}\right)\right\}$$
with $\overline p(\cdot)$ as in Theorem \ref{W2qexterior}, and $u$ is the unique solution of problem \eqref{PF}, then $\nabla u$ is H\"older continuous in $[\ve,T]\times\ov\O$ and its parabolic seminorm \eqref{holderseminorm} is evaluated by
$$[\nabla u]_\lambda\le \frac c{\ve^\beta}\|u_0\|_2^\alpha$$
where 
$$\theta=\frac{n-2+2\lambda}{n+2+2\lambda},\quad \alpha=\max\left\{(2-p)\gamma+1,\frac2 p,(2-p)\gamma\theta+1\right\},\quad \beta=\max\{1+\gamma,p'\}$$
and $\gamma=\gamma\left(\frac n{n-1+\lambda}\right)$ (see \eqref{Maxpostraa}).
\end{theorem}

\section{Existence with data in $L^s(\O)$}\label{Lstheory}

In this section we investigate the existence of a solution of problem \eqref{PF} when the initial data are in $L^s(\O)$ with $1\le s<2$. To this purpose we need to adapt the definition of solution to the new framework in the following way
\begin{definition}\label{weaksolutionLs}
Let $\O$ be a bounded or exterior domain with boundary of class $C^2$ and $u_{\circ}\in L^s(\O)$, $1\le s<2$. A field
$u\!:(0,T)\times \O\to\R^N$
 is said a solution of
system {\rm \eqref{PF}} if $$ u\in L^{\infty}(0,T;L^s(\O))\cap
L^{\frac{ps}2}(0,T;\widehat W_0^{1,\frac{ps}2}(\O))\,, $$ 
\be\label{testforLs}\int_0^T [(u,\psi_t)-\left(|\nabla u|^{p-2}\,
\nabla u,\nabla \psi\right)]\,dt=
-(u_\circ, \psi(0)),\qquad \forall \psi\in
C^\infty_0([0,T)\times \O)\ee
and
$$\lim_{t\to 0^+}\|u(t)-u_\circ\|_s=0\,.$$
\end{definition}

%As the reader can easily observe, we are considering only the case of a bounded domain ant this is due to the fact that we are able to prove the existence of a solution only for this kind of domains. Despite of this fact, since this paper is mainly concerned with exterior domains, we state some preliminary estimates that are valid also in the case of unbounded sets.
%
%\begin{lemma}
%Let $\O$ be a bounded or exterior domain of $\R^n$ of class $C^2$. For any $u_\circ\in L^s(\O)$ let be $\{u_\circ^k\}$ a sequence of functions in $C^\infty_0(\O)$ approximating $u_\circ$ in $L^s(\O)$. For any $k$ let be $u^k$ the unique solution of problem \eqref{PF} provided by Theorem \ref{existencebounded} or Theorem \ref{existenceexterior}. Then we have that there exists a function $u$ such that
%$$u^k\mathop{\rightharpoonup}^*u\hbox{ weakly}^*\ \hbox{in }L^\infty(0,T;L^s(\O)),\qquad\nabla u^k\rightharpoonup\nabla u\hbox{ weakly in }L^{\frac{ps}2}(0,T;L^{\frac{ps}2}(\O))$$
%$$\|u\|_{L^\infty(0,T;L^s(\O))}\le\|u_\circ\|_s,\qquad \|\nabla u\|_{L^{\frac{ps}2}(0,T;L^{\frac{ps}2}(\O))}\le c\|u_\circ\|_s^{\frac2p}$$
%and, for any $\ve\in(0,T)$
%$$u^k\mathop{\rightharpoonup}^*u\hbox{ weakly}^*\ \hbox{in }L^\infty(\ve,T;L^2(\O)),\qquad \nabla u^k\rightharpoonup\nabla u\hbox{ weakly in }L^p(\ve,T;L^p(\O)).$$
%$$\|u^k(t)\|_2\le \frac c{\ve^{\frac1a}}\,\|u_\circ\|_s^{\frac{s((n-p)(\bar s-2)+2-s)}{2(n-p)(\bar s-s)}}$$
%\end{lemma}

Before we state the main theorem of this section, we need to define a number which will be crucial for the existence of the solution. Let be
$$\sex=\left\{\ba{ll}n\left(\frac2p-1\right)&\mbox{if }\frac{2n}{n+2}<p<\frac{2n}{n+1}\\
1&\mbox{if }\frac{2n}{n+1}\le p<2.\ea\right.$$

\begin{theorem}\label{Lsexistence}
Let $\O$ be a bounded or exterior domain of class $C^2$ and $p\in\left(\frac{2n}{n+2},2\right)$. If $u_\circ\in L^s(\O)$ with $s\in(\sex,2)$, then there exists a solution $u$ of problem \eqref{PF}
in the sense of Definition \ref{weaksolutionLs}. Moreover we have that, for any $\ve>0$ the estimates \eqref{LinfinityL2}-\eqref{tp+2nablaut} hold true in the interval $t\in[\ve,T]$ assuming $u(\ve)$ as initial data in place of $u_\circ$ and, for suitable $\gamma,\alpha>0$
\be\label{L2leqLs}\|u(t)\|_2\le c\,\frac{\|u_\circ\|_s^\alpha}{t^\gamma}\qquad\forall t>0.\ee
\end{theorem}

\begin{proof}
Let be $\O$ an exterior domain and, for any $k\in\N$, $E_k=\O\cap B(0,k)$. We take $k$ large enough to have $(\R^n\setminus\O)\subset B(0,k)$. If $\O$ is a bounded domain we simply take $E_k=\O$ for any $k$.
We can find a sequence $\{u^k_\circ\}\subset C^\infty_0(E_k)$ converging to $u_\circ$ in $L^s(\O)$. 
Since $u_\circ^k\in L^2(E_k)$, by Theorem \ref{existencebounded}, there exists a unique solution of problem \eqref{PF} in $(0,T)\times E_k$, corresponding to the initial data $u_\circ^k$, that we denote by $u^k$. Following Remark \ref{testsmooth} we fix $\delta>0$ and  we use $(\delta+|u^k|^2)^{\frac{s-2}2}u^k$ as a test function in equation \eqref{testW1p}. Integrating by parts we get
$$\int_0^t\langle u^k_t,(\delta+|u^k|^2)^{\frac{s-2}2}u^k\rangle\,d\tau+\int_0^t\int_{E_k}|\nabla u^k|^{p-2}\nabla u^k\cdot\nabla\left((\delta+|u^k|^2)^{\frac{s-2}2}u^k\right)\,dx\,d\tau=0$$
hence
$$\ba{l}\dy\vs1\frac1s\|(\delta+|u^k(t)|^2)^\frac12\|_s^s+\int_0^t\int_{E_k}|\nabla u^k|^p(\delta+|u^k|^2)^{\frac{s-2}2}\,dx\,d\tau\\
\hfill\dy+(s-2)\int_0^t\int_{E_k}|\nabla u^k|^{p-2}|\nabla u^k\cdot u^k|^2(\delta+|u^k|^2)^{\frac{s-4}2}\,dx\,d\tau
=\frac1s\|(\delta+|u_\circ^k|^2)^\frac12\|_s^s.\ea$$
Since
$$|\nabla u^k|^{p-2}|\nabla u^k\cdot u^k|^2(\delta+|u^k|^2)^{\frac{s-4}2}\le|\nabla u^k|^p(\delta+|u^k|^2)^{\frac{s-2}2}$$
we have that
\be\label{nrgy}\|(\delta+|u^k(t)|^2)^\frac12\|_s^s+s(s-1)\int_0^t\int_{E_k}|\nabla u^k|^p(\delta+|u^k|^2)^{\frac{s-2}2}\,dx\,d\tau\le\|(\delta+|u_\circ^k|^2)^\frac12\|_s^s.\ee
Since $1<s<2$ it follows that
\be\label{ukLinftyLs}\|u^k(t)\|_s^s\le\|(\delta+|u^k(t)|^2)^\frac12\|_s^s\le\delta^{\frac s2}|E_k|+\|u_\circ^k\|_s^s.\ee
We can apply the Inequality \ref{reverse} with exponents $\frac s2$ and $\frac s{s-2}$ to obtain
$$\int_{E_k}|\nabla u^k|^p(\delta+|u^k|^2)^{\frac{s-2}2}\,dx\ge\left(\int_{E_k}|\nabla u^k|^{\frac{ps}2}\,dx\right)^{\frac2s}\left(\int_{E_k}\left(\delta+|u^k|^2\right)^{\frac s2}\,dx\right)^{\frac{s-2}s}.$$
Hence, by \eqref{nrgy}
$$\ba{l}\dy\vs1\int_{E_k}|\nabla u^k|^{\frac{ps}2}\,dx\le\left(\int_{E_k}|\nabla u^k|^p(\delta+|u^k|^2)^{\frac{s-2}2}\,dx\right)^{\frac s2}\left(\int_{E_k}\left(\delta+|u^k|^2\right)^{\frac s2}\,dx\right)^{\frac{2-s}2}\\
\hfill\dy\le c\left(\int_{E_k}|\nabla u^k|^p(\delta+|u^k|^2)^{\frac{s-2}2}\,dx\right)^{\frac s2}\left\|\left(\delta+|u_\circ^k|^2\right)^{\frac12}\right\|_s^{\frac{s(2-s)}2}\ea.$$
Integrating in time, by means of the H\"older inequality and \eqref{nrgy}, we have
\be\label{boundLps2Lps2}\int_0^t\int_{E_k}|\nabla u^k|^{\frac{ps}2}\,dx\,d\tau
\le c\,\left\|\left(\delta+|u_\circ^k|^2\right)^{\frac12}\right\|_s^s\ t^{\frac{2-s}2}\le c\left(\delta^{\frac s2}|E_k|+\|u_\circ^k\|_s^s\right)t^{\frac{2-s}2}.\ee
Since the sequence $\{u_\circ^k\}$ converges to $u_\circ$ in $L^s(\O)$, we have that $\|u_\circ^k\|_s^s\le c\|u_\circ\|_s^s$ and, letting $\delta\to0$, by \eqref{ukLinftyLs} we get
\be\label{boundLinftyLs}\|u^k(t)\|_s^s\le \|u^k_\circ\|_s^s\le c\|u_\circ\|_s^s\quad\forall k\in\N,\ \forall t\in[0,T]\ee
and, by \eqref{boundLps2Lps2}, also
\be\label{nablauks-22}\int_0^t\int_{E_k}|\nabla u^k|^{\frac{ps}2}\,dx\,d\tau\le c\|u_\circ\|_s^s,\qquad\forall t\in[0,T].\ee
Extending to 0 the functions $u^k$ in $\O\setminus E_k$ and using the uniform bounds \eqref{boundLinftyLs},\eqref{nablauks-22} we can extract a subsequence (not relabeled) and find a function $u$ such that
\be\label{weakLsLps2}u^k\mathop{\rightharpoonup}^*u\hbox{ weakly}^*\ \hbox{in }L^\infty(0,T;L^s(\O)),\qquad\nabla u^k\rightharpoonup\nabla u\hbox{ weakly in }L^{\frac{ps}2}(0,T;L^{\frac{ps}2}(\O)).\ee
By the strong convergence of $u_\circ^k$ towards $u_\circ$ in $L^s(\O)$ and by the weak convergence of $\nabla u^k$, letting $k\to\infty$ in \eqref{boundLinftyLs} and in \eqref{nablauks-22} we get
\be\label{uLinftyLs}\|u\|_{L^\infty(0,T;L^s(\O))}\le\|u_\circ\|_s,\qquad \|\nabla u\|_{L^{\frac{ps}2}(0,T;L^{\frac{ps}2}(\O))}\le c\|u_\circ\|_s^{\frac2p}.\ee
Let us define $v^k=u^k(\delta+|u^k|^2)^{\frac{s-2}{2p}}$. A straightforward computation shows that
\be\label{nablavknablauk}|\nabla v^k|^p\le\left(1+\left(\frac{2-s}{2p}\right)^p\right) |\nabla u^k|^p\left(\delta+|u^k|^2\right)^{\frac{s-2}2}.\ee
The Sobolev's inequality, \eqref{nablavknablauk} and \eqref{nrgy} lead to
$$\int_0^t\left(\int_{E_k}|v^k|^{p^*}\,dx\right)^{\frac{n-p}n}\,d\tau\le c\int_0^t\int_{E_k}|\nabla v^k|^p\,dx\,d\tau\le c\left\|(\delta+|u_\circ^k|^2)^\frac12\right\|_s^s.$$
In terms of $u_k$ the above inequality becomes
$$\int_0^t\left(\int_{E_k}|u^k|^{\frac{np}{n-p}}\left(\delta+|u^k|^2\right)^{\frac{(s-2)n}{2(n-p)}}\,dx\right)^{\frac{n-p}n}\,d\tau\le c\left(\delta^{\frac s2}|E_k|+\|u_\circ^k\|_s^s\right).$$
Since $s<2$, letting $\delta\to0$, by monotone convergence we get
$$\int_0^t\left(\int_{E_k}|u^k|^{\frac{n(p+s-2)}{n-p}}\,dx\right)^{\frac{n-p}n}\,d\tau\le c\|u_\circ\|_s^s.$$
If we set $s_1=\frac{n(p+s-2)}{n-p}$ we have that $u^k\in L^{\frac{s_1(n-p)}n}(0,T;L^{s_1}(E_k))$ uniformly in $k$. We have that
$$s_1-s=\frac{p(s+n)-2n}{n-p}$$ 
and if $s>\sex$ then $s_1-s>0$ hence, iterating the process, we obtain an increasing sequence $\{s_m\}$ such that $s_{m+1}-s_m>s_m-s_{m-1}$. In a finite number of steps we get that $u^k\in L^{\frac{\bar s(n-p)}n}(0,T;L^{\bar s}(E_k))$ with $\bar s\ge2$. We remark that, if $p>\frac{2n}{n+1}$ then $s_1-s>0$ for any $s\in(1,2)$ and we have no restrictions on $s$. On the contrary, if $p\le\frac{2n}{n+2}$ then $s_1-s\le0$ for any $s\in(1,2)$ and the iteration is useless. In the end we get
\be\label{LinftyLbars}\int_0^t\|u^k\|_{\bar s}^{\frac{\bar s(n-p)}n}\,d\tau\le c\|u_\circ\|_s^s\ee
and, since $\bar s\ge2$ we can interpolate $L^2$ between $L^s$ and $L^{\bar s}$ obtaining 
\be\label{interpol}\|u^k\|_2\le\|u^k\|^{1-\vartheta}_s\,\|u^k\|_{\bar s}^\vartheta,\qquad \vartheta=\frac{(2-s)\bar s}{2(\bar s-s)}.\ee
It we set $a=\frac{\bar s(n-p)}{n\vartheta}=\frac{2(n-p)(\bar s-s)}{2-s}$, by \eqref{interpol}, \eqref{boundLinftyLs} and \eqref{LinftyLbars} we get
\be\label{LaL2}\int_0^t\|u^k\|_2^a\le\sup_\tau\|u^k(\tau)\|_s^{a(1-\vartheta)}\int_0^t\|u^k\|_{\bar s}^{a\vartheta}\le c\|u_\circ\|_s^{\frac{s((n-p)(\bar s-2)+2-s)}{2-s}}.\ee
Now we go back to the equation \eqref{testW1p}, we use $u^k$ as a test function and we differentiate with respect to $t$ getting
\be\label{energyuk}\frac12\frac d{dt}\|u^k(t)\|_2^2+\|\nabla u^k(t)\|_p^p=0\quad \forall\, t\in[0,T].\ee
By the above equality it follows that, if there exists $\bar t$ such that $\|u^k(\bar t)\|_2=0$ then $\|u^k(t)\|_2=0$ for any $t\ge\bar t$ hence we can suppose $\|u^k(t)\|_2>0$ for any $t$, otherwise what follows is trivially true.
Multiplying by $t$ we have
$$\ba{l}\dy\vs1 0\ge\frac t2\frac d{dt}\|u^k(t)\|_2^2=t\|u^k(t)\|_2\frac d{dt}\|u^k(t)\|_2=t\|u^k(t)\|^{2-a}_2\left(\|u^k(t)\|_2^{a-1}\frac d{dt}\|u^k(t)\|_2\right)\\
\hfill\dy=\frac ta\|u^k(t)\|_2^{2-a}\frac d{dt}\|u^k(t)\|_2^a\ea$$
and, multiplying by $a\|u^k(t)\|_2^{a-2}$
$$0\ge t \frac d{dt}\|u^k(t)\|_2^a=\frac d{dt}\left(t\|u^k(t)\|_2^a\right)-\|u^k(t)\|_2^a.$$
Integrating the above inequality, by \eqref{LaL2}, we have
\be\label{LinftyL2epsilonT}\|u^k(t)\|_2\le \frac c{t^{\frac1a}}\,\|u_\circ\|_s^{\frac{s((n-p)(\bar s-2)+2-s)}{2(n-p)(\bar s-s)}}\quad\forall\, t\in(0,T]\ee
with the constant $c$ not depending on $k$. Fixing $\ve>0$ and integrating in time the identity \eqref{energyuk}, by \eqref{LinftyL2epsilonT} we get
\be\label{LpW1pepsilonT}\int_\ve^T\int_{E_k}|\nabla u^k(\tau)|^p\,dx\,d\tau\le\frac c{\ve^{\frac2a}}\,\|u_\circ\|_s^{\frac{s((n-p)(\bar s-2)+2-s)}{(n-p)(\bar s-s)}}.\ee
Estimates \eqref{LinftyL2epsilonT} and \eqref{LpW1pepsilonT} are enough to say that, up to a subsequence
\be\label{weakLinftyL2}u^k\mathop{\rightharpoonup}^*u\hbox{ weakly}^*\ \hbox{in }L^\infty(\ve,T;L^2(\O)),\ee
$$\nabla u^k\rightharpoonup\nabla u\hbox{ weakly in }L^p(\ve,T;L^p(\O)).$$
We remark that the limit point of above convergences is actually $u$ by the convergences \eqref{weakLsLps2} and the uniqueness of the weak limit. 
\par
By \eqref{LinftyL2epsilonT} and up to further subsequences, we can find $\Theta,\Lambda\in L^2(\O)$ such that
\be\label{weakepsilonT}u^k(\ve)\rightharpoonup \Theta,\qquad u^k(T)\rightharpoonup\Lambda\qquad\mbox{weakly in }L^2(\O).\ee
The estimate \eqref{LpW1pepsilonT} allows to find a function $\chi\in L^{p'}(\ve,T;L^{p'}(\O))$ such that, up to a subsequence
\be\label{weakchi}|\nabla u^k|^{p-2}\nabla u^k\rightharpoonup \chi\quad\mbox{weakly in }L^{p'}(\ve,T;L^{p'}(\O)).\ee
Let us choose $\psi\in C^\infty_0((\ve,T)\times\O)$ and $k$ large enough to have $\psi(t,x)=0$ for any $(t,x)\in(\ve,T)\times(\O\setminus E_k)$. Since $u^k$ is a solution on $E_k$ we can use $\psi$ as a test function in \eqref{testW12} (substituting $u^k$ in place of $u$ and $E_k$ in place of $\O$). Integrating by parts we get
$$\int_\ve^T(u^k,\psi_t)\,dt-\int_\ve^T\left(|\nabla u^k|^{p-2}\nabla u^k,\nabla\psi\right)\,dt=0.$$
Passing to the limit as $k\to\infty$, thanks to \eqref{weakLinftyL2} and \eqref{weakchi}, we have
$$\int_\ve^T(u,\psi_t)\,dt-\int_\ve^T(\chi,\nabla\psi)\,dt=0\qquad\forall\psi\in C^\infty_0\left((\ve,T)\times \O\right).$$
As a consequence of the above identity $u_t\in L^{p'}(\ve,T;\widehat W^{1,p}_0(\O)')$ and, by density
\be\label{testLpV}\int_\ve^T\langle u_t,\psi\rangle\,dt+\int_\ve^T(\chi,\nabla\psi)\,dt=0\qquad\forall\psi\in L^p(\ve,T;\widehat W_0^{1,p}(\O)).\ee
Now we take $\psi\in C^\infty_0\left([\ve,T]\times \O\right)$ in \eqref{testW1p} to get
\be\label{u^ktest[epT]}\int_\ve^T(u^k,\psi_t)-\left(|\nabla u^k|^{p-2}\nabla u^k,\nabla\psi\right)\,dt=\left(u^k(T),\psi(T)\right)-\left(u^k(\ve),\psi(\ve)\right).\ee
Passing to the limit for $k\to\infty$ and remembering \eqref{weakepsilonT} we obtain
\be\label{ThetaLambda}\int_\ve^T(u,\psi_t)-(\chi,\nabla\psi)\,dt=\left(\Lambda,\psi(T)\right)-\left(\Theta,\psi(\ve)\right).\ee
Since $u\in C([\ve,T];L^2(\O))$ we can integrate \eqref{testLpV} by parts to gain
\be\label{uepsilonuT}-\int_\ve^T(u,\psi_t)\,dt+\left(u(T),\psi(T)\right)-\left(u(\ve),\psi(\ve)\right)+\int_\ve^T(\chi,\nabla\psi)\,dt=0.\ee
Comparing the identity \eqref{uepsilonuT} with \eqref{ThetaLambda}, by the arbitrariness of $\psi$ we get
\be\label{Theta=u(ve)}u(\ve)=\Theta,\qquad u(T)=\Lambda.\ee
Now we choose an arbitrary function $\psi\in C^\infty_0\left([\ve,T)\times\O\right)$ and we use it in equation \eqref{u^ktest[epT]} obtaining
\be\label{u^kpsiepT}\int_\ve^T(u^k,\psi_t)-\left(|\nabla u^k|^{p-2}\nabla u^k,\nabla\psi\right)\,dt=-\left(u^k(\ve),\psi(\ve)\right).\ee
If we apply Theorem \ref{W22IN} to the function $u^k$ on $\left(\frac\ve2,T\right)\times E_k$, using $u^k(\frac\ve2)$ as initial data, we get
$$\|D^2u^k(t)\|_2\le\frac c{\left(t-\frac\ve2\right)^\gamma}\left\|u^k\left(\frac\ve2\right)\right\|_2^\alpha,\qquad \forall t\in\left(\frac\ve2,T\right]$$
for suitable $\alpha,\gamma >0$. We remark that the constant $c$ does not depend on $k$. This is not totally trivial but it is a consequence of Theorem \ref{CGMThm1.2} and \cite[Corollary 3.1]{CGM} and it relies on the fact that $c$ depends on the geometric properties of the boundary of $E_k$ and not on its measure. Now we use \eqref{LinftyL2epsilonT} with $t=\ve$ to get
\be\label{D2estimateonK}\|D^2u^k(t)\|_2\le \frac c{\ve^{\gamma_1}}\|u_\circ\|_s^{\alpha_1},\qquad\forall t\in[\ve,T]\ee
for suitable $\alpha_1,\gamma_1>0$ and $c$ not depending on $k$. Now let be $K\subset\R^n$ an open bounded set such that $\psi(t,x)=0$ for any $(t,x)\in[\ve,T]\times(\R^n\setminus K)$. By estimate \eqref{D2estimateonK} we have that 
$u^k\in L^\infty(\ve,T,W^{2,2}(K))$ uniformly in $k$ and, by the Rellich-Kondrachov theorem,
$$\nabla u^k(t,x)\longrightarrow \nabla u(t,x)\ \mbox{a.e. in } [\ve,T]\times K$$
up to a subsequence.
Since, by \eqref{LpW1pepsilonT}
$$\int_\ve^T\int_K\left||\nabla u^k|^{p-2}\nabla u^k\right|^{p'}\,dx\,dt\le c$$
uniformly in $k$, we can apply \cite[Lemma I.1.3]{lions} to obtain that 
\be\label{convnablauktonablau}|\nabla u^k|^{p-2}\nabla u^k\rightharpoonup |\nabla u|^{p-2}\nabla u\quad\mbox{weakly in } L^{p'}(\ve,T,L^{p'}(K))\ee
and, by \eqref{weakchi}
$$\chi=|\nabla u|^{p-2}\nabla u\qquad\mbox{for a.e. }(t,x)\in(\ve,T)\times K.$$
We remark once again that the function $u$ is defined by the global weak convergences \eqref{weakLsLps2} hence, by the arbitrariness of $\psi$ and $\ve$, we get that the above identity holds  almost everywhere in $(0,T)\times\O$. Passing to the limit for $k\to\infty$ in equation \eqref{u^kpsiepT} with the aid of \eqref{convnablauktonablau} we have
\be\label{equationforuoneT}\int\limits_\ve^T(u,\psi_t)-\left(|\nabla u|^{p-2}\nabla u,\nabla\psi\right)\,dt=-\left(u (\ve),\psi(\ve)\right).\ee
Now we complete the existence proof taking a test function $\psi\in C^\infty_0\left([0,T)\times\O\right)$. 
By equation \eqref{equationforuoneT} we get
\be\label{equationforuon0T}\int_0^T(u,\psi_t)-\left(|\nabla u|^{p-2}\nabla u,\nabla\psi\right)\,dt=\int_0^\ve(u,\psi_t)-\left(|\nabla u|^{p-2}\nabla u,\nabla\psi\right)\,dt-\left(u(\ve),\psi(\ve)\right).\ee
Using the continuity of the Lebesgue integral with respect to the domain of integration we get that the first integral on the right-hand side of the above identity vanishes as $\ve\to0$. It resmains to estimate the term $\left(u(\ve),\psi(\ve)\right)$. We have
\be\label{A+B+C}\ba{l}\vs1\dy\left|\left(u(\ve),\psi(\ve)\right)-\left(u_\circ,\psi(0)\right)\right|\\
\hfill\dy\vs1\le\left|\left(u(\ve)-u^k(\ve),\psi(\ve)\right)\right|
+\left|\left(u^k(\ve),\psi(\ve)\right)-\left(u_\circ^k,\psi(0)\right)\right|
+\left|\left(u^k_\circ-u_\circ,\psi(0)\right)\right|\\
\hfill\dy=:A_k(\ve)+B_k(\ve)+C_k\ea\ee
By \eqref{weakepsilonT} and \eqref{Theta=u(ve)} we get, for any $\ve>0$
$$\lim_{k\to\infty} A_k(\ve)=0.$$
Since $u^k$ is a solution in $E_k$, for any $k$ large enough to contain the spatial support of $\psi$, we have (see \eqref{testsmoothst})
$$\ba{l}\dy\vs1B_k(\ve)\le\int_0^\ve\left|(u^k,\psi_t)-\left(|\nabla u^k|^{p-2}\nabla u^k,\nabla\psi\right)\right|\,dt\\
\hfill\dy\vs1\le \|u^k\|_{L^\infty(0,T;L^s(\O))}\|\psi_t\|_{L^1(0,\ve;L^{s'}(\O))}\\
\hfill\dy\vs1+\|\nabla u^k\|_{L^{\frac{ps}2}(0,T;L^{\frac{ps}2}(\O))}^{p-1}\|\nabla\psi\|_{L^{\left(\frac{ps}{2(p-1)}\right)'}(0,\ve;L^{\left(\frac{ps}{2(p-1)}\right)'}(\O))}\\
\hfill\dy\le B(\ve)
\ea$$
where, by \eqref{boundLinftyLs} and \eqref{nablauks-22}, $B(\ve)$ is a function not depending on $k$ and infinitesimal as $\ve\to0$.
Finally, 
$$C_k\le\|u_\circ^k-u_\circ\|_s\|\psi(0)\|_{s'}\longrightarrow 0\quad\mbox{for }k\to\infty$$
by hypothesis on $u_\circ^k$. Passing to the limit for $k\to\infty$ in \eqref{A+B+C} we get
\be\label{B(eps)}\left|\left(u(\ve),\psi(\ve)\right)-\left(u_\circ,\psi(0)\right)\right|\le B(\ve)\ee
and passing to the limit for $\ve\to0$ in \eqref{equationforuon0T} we get that $u$ is a solution with initial data $u_\circ\in L^s(\O)$.

%%%%%

It remains to prove that the initial datum is assumed strongly in $L^s(\O)$.
Let us fix $\vp\in C^\infty_0(\O)$, $0<\delta<T$ and set $\Phi(x,t)=\vp(x)\theta(t)$ with $\theta\in C^\infty\left([0,T)\right)$ and $\theta(t)=1$ for any $t\in[0,\delta]$. Using $\Phi$ as a test function and reasoning exactly as in the evaluation of $B_k(\ve)$, by \eqref{B(eps)} we can get
$$\left|\left(u(t)-u_\circ,\vp\right)\right|\le B(t)$$
for any $t\in[0,\delta]$, hence
$$\lim_{t\to0^+}\left(u(t)-u_\circ,\vp\right)=0\qquad\forall\vp\in C^\infty_0(\O).$$
The density of $C^\infty_0(\O)$ in $L^s(\O)$ gives the weak convergence in $L^s(\O)$ of $u(t)$ to $u_\circ$. By lower semicontinuity we also get
$$\|u_\circ\|_s\le\liminf_{t\to0^+}\|u(t)\|_s.$$
By \eqref{uLinftyLs} we have 
$$\limsup_{t\to0^+}\|u(t)\|_s\le\|u_\circ\|_s$$
hence 
$$\lim_{t\to0^+}\|u(t)\|_s=\|u_\circ\|_s$$
and the uniform convexity of $L^s(\O)$ gives the strong convergence
$$\lim_{t\to0^+}\|u(t)-u_\circ\|_s=0.$$
Finally, the estimate \eqref{L2leqLs} follows passing to the limit as $k\to\infty$ in \eqref{LinftyL2epsilonT} using \eqref{weakLinftyL2} and lower-semicontinuity.
\end{proof}

\section{Extinction of the solutions}\label{extinction}

\begin{proof}[Proof of Theorem \ref{EXTIN}]
Let us consider, for any $R>0$, the smooth cut-off function $h_R$ defined in \eqref{hR}.
Let us fix $\sigma\in(\sex,s]$ and consider the solution $u$ solution obtained in Theorem \ref{Lsexistence} with initial data $u_\circ\in L^\sigma(\O)$. Then, for any $\ve>0$ we have that $u$ solves equation \eqref{testW1p} in $(\ve,T)\times\O$ hence we can differentiate it with respect to $t$ obtaining
\be\label{differential2.5}\langle u_t,\psi\rangle=-\left(|\nabla u|^{p-2}\nabla u,\nabla\psi\right)\ \forall \psi\in V^{p,p'}(\ve,T;\O).\ee
For any $\delta>0, $ and $R$ suitably large, we have that
$$h_R u(|u|^2+\delta)^{\frac{\sigma-2}2}\in V^{p,p'}(\ve,T;\O)$$
hence we can use it as a test function in \eqref{differential2.5} obtaining
$$\ba{l}\dy\vs1-\frac1\sigma\frac d{dt}\left\|h_R^{1/\sigma}\left(|u|^2+\delta\right)^{1/2}\right\|_\sigma^\sigma
=\int_\O|\nabla u|^{p-2}\nabla u\cdot(\nabla h_R\otimes u)\left(|u|^2+\delta\right)^{\frac{\sigma-2}2}\,dx\\
\hfill\dy+\int_\O|\nabla u|^ph_R\left(|u|^2+\delta\right)^{\frac{\sigma-2}2}\,dx
+(\sigma-2)\int_\O|\nabla u|^p h_R |u|^2 \left(|u|^2+\delta\right)^{\frac{\sigma-4}2}\,dx.\ea
$$
Since $\sigma\le2$, setting $E_{R,2R}=\left\{x\in\O:R<|x|<2R\right\}$, we have
\be\label{diffineq}\ba{l}\dy\vs1\frac1\sigma\frac d{dt}\left\|h_R^{1/\sigma}\left(|u|^2+\delta\right)^{1/2}\right\|_\sigma^\sigma
+(\sigma-1)\int_\O|\nabla u|^p h_R\left(|u|^2+\delta\right)^{\frac{\sigma-2}2}\,dx\\
\hfill\dy\le\frac cR\int_{E_{R,2R}}|\nabla u|^{p-1}\left(|u|^2+\delta\right)^{\frac{\sigma-1}2}\,dx.\ea\ee
Applying the H\"older inequality with exponents $\frac{p\sigma}{2(p-1)},\ \frac\sigma{\sigma-1}$ and $\frac{\sigma p}{2-p}$ on the integral on the right-hand side of \eqref{diffineq} we get
\be\label{right}\int_{E_{R,2R}}|\nabla u|^{p-1}\left(|u|^2+\delta\right)^{\frac{\sigma-1}2}\,dx
\le c\|\nabla u\|_{L^\frac{p\sigma}2(E_{R,2R})}^{p-1}\left\|\left(|u|^2+\delta\right)^{1/2}\right\|_{L^\sigma(E_{R,2R})}^{\sigma-1}R^{\frac{n(2-p)}{\sigma p}}.\ee
Using Inequality \ref{reverse} with exponents $\frac\sigma2$ and $\frac\sigma{\sigma-2}$ on the integral on the left-hand side of \eqref{diffineq} we have
\be\label{left}\ba{l}\dy\vs1\int_\O|\nabla u|^ph_R\left(|u|^2+\delta\right)^{\frac{\sigma-2}2}\,dx=\int_\O|\nabla u|^ph_R^{3-\sigma}\left(\left(|u^2|+\delta\right)h^2_R\right)^{\frac{\sigma-2}2}\,dx\\
\hfill\dy\ge\left\||\nabla u|h_R^{\frac{3-\sigma}p}\right\|_{\frac{p\sigma}2}^p\left\|\left(|u|^2+\delta\right)^{1/2}h_R\right\|_\sigma^{\sigma-2}
\ge\left\||\nabla u|h_R^{\frac{3-\sigma}p}\right\|_{\frac{p\sigma}2}^p\left\|\left(|u|^2+\delta\right)^{1/2}h_R^{1/\sigma}\right\|_\sigma^{\sigma-2}
.\ea\ee
Now we remark that
$$\frac d{dt}\left\|\left(|u|^2+\delta\right)^{1/2}h_R^{1/\sigma}\right\|_\sigma^2=\frac2\sigma\left\|\left(|u|^2+\delta\right)^{1/2}h_R^{1/\sigma}\right\|_\sigma^{2-\sigma}\frac d{dt}\left\|\left(|u|^2+\delta\right)^{1/2}h_R^{1/\sigma}\right\|_\sigma^\sigma$$
hence we multiply inequality \eqref{diffineq} by $\left\|\left(|u|^2+\delta\right)^{1/2}h_R^{1/\sigma}\right\|_\sigma^{2-\sigma}$ obtaining, by \eqref{right} and \eqref{left}
$$\ba{l}\dy\vs1\frac12\frac d{dt}\left\|\left(|u|^2+\delta\right)^{1/2}h_R^{1/\sigma}\right\|_\sigma^2+(\sigma-1)\left\||\nabla u|h_R^{\frac{3-\sigma}p}\right\|_{\frac{p\sigma}2}^p\\
\hfill\dy\le c\|\nabla u\|_{L^{\frac{p\sigma}2}(E_{R,2R})}^{p-1}\left\|\left(|u|^2+\delta\right)^{1/2}\right\|_{L^\sigma(E_{2R})} R^{\frac{n(2-p)}{\sigma p}-1}.\ea$$
Let us integrate in time the above inequality to get
$$\ba{l}\dy\vs1\left\|\left(|u(t)|^2+\delta\right)^{1/2}h_R^{1/\sigma}\right\|_\sigma^2+2(\sigma-1)\int_\ve^t\left\||\nabla u|h_R^{\frac{3-\sigma}p}\right\|_{\frac{p\sigma}2}^p\,d\tau\\
\hfill\dy\le cR^{\frac{n(2-p)}{\sigma p}-1}\int\limits_\ve^t \|\nabla u\|_{L^{\frac{p\sigma}2}(E_{R,2R})}^{p-1}\left\|\left(|u|^2+\delta\right)^{1/2}\right\|_{L^\sigma(E_{2R})}d\tau +\left\|\left(|u(\ve)|^2+\delta\right)^{1/2}h_R^{1/\sigma}\right\|_\sigma^2.\ea$$
Keeping count that all the above integrals are evaluated on the bounded set $E_{2R}$, we can apply the dominated convergence theorem letting $\delta\to0$ in order to obtain
\be\label{hRsigma}\ba{l}\dy\vs1\left\|u(t)h_R^{1/\sigma}\right\|_\sigma^2+2(\sigma-1)\int_\ve^t\left\||\nabla u|h_R^{\frac{3-\sigma}p}\right\|_{\frac{p\sigma}2}^p\,d\tau\\
\hfill\dy\le cR^{\frac{n(2-p)}{\sigma p}-1}\int_\ve^t \|\nabla u\|_{L^{\frac{p\sigma}2}(E_{R,2R})}^{p-1}\left\|u\right\|_{L^\sigma(E_{2R})}\,d\tau +\left\|u(\ve)h_R^{1/\sigma}\right\|_\sigma^2.\ea\ee
We remark that
$$\frac{n(2-p)}{\sigma p}-1<0 \iff \sigma>n\left(\frac2p-1\right)=\sex$$
hence, passing to the limit as $R\to\infty$ in \eqref{hRsigma} we have
\be\label{usigmaepsilon}\|u(t)\|_\sigma^2+2(\sigma-1)\int_\ve^t\|\nabla u\|_{\frac{p\sigma}2}^p\,d\tau\le\|u(\ve)\|_\sigma^2.\ee
Using the strong continuity of $u$ in $L^\sigma$ for $t=0$ (see Theorem \ref{Lsexistence}) we can pass to the limit for $\ve\to0$ in the above inequality obtaining
\be\label{usigmau0sigma}\|u(t)\|_\sigma^2+2(\sigma-1)\int_0^t\|\nabla u\|_{\frac{p\sigma}2}^p\,d\tau\le\|u_\circ\|_\sigma^2.\ee
Now we remark that, being the right-hand side of \eqref{usigmau0sigma} bounded, we have
\be\label{sigmatosex}\lim_{\sigma\to\sex}\|u(t)\|_\sigma=\|u(t)\|_\sex,\quad\lim_{\sigma\to\sex}\|u_\circ\|_\sigma=\|u_\circ\|_\sex,\quad \lim_{\sigma\to\sex}\|\nabla u(t)\|_{\frac{p\sigma}2}=\|\nabla u(t)\|_{\frac{p\sex}2}\ee
and, by Fatou lemma
$$\int_\ve^t\|\nabla u\|_{\frac{p\sex}2}^p\,d\tau\le\liminf_{\sigma\to\sex}\int_\ve^t\|\nabla u\|_{\frac{p\sigma}2}^p\,d\tau.$$
Passing to the $\liminf$ in inequality \eqref{usigmaepsilon} we get
\be\label{liminf}\|u(t)\|_\sex^2+2(\sex-1)\int_\ve^t\|\nabla u\|_{\frac{p\sex}2}^p\,d\tau\le\|u(\ve)\|_\sex^2.\ee
By means of the Sobolev inequality, observing that $\left(\frac{p\sex}2\right)^*=\sex$, we have
\be\label{diffinequal}\|u(t)\|_\sex^2+c\int_\ve^t\|u(\tau)\|_\sex^p\,d\tau\le\|u(\ve)\|_\sex^2.\ee
Now we set $w(t)=\|u(t)\|_\sex^2$ obtaining
\be\label{integrineqw}w(t)+c\int_\ve^t w(\tau)^{\frac p2}\,d\tau\le w(\ve),\qquad \forall\,0\le\ve\le t.\ee
Let us consider the Cauchy problem
$$\left\{\ba{l}z'+cz^{\frac p2}=0\\
z(0)=\|u_\circ\|_\sex^2\\
z(t)\ge0\ea\right.$$
whose solution is $z(t)=\left(\|u_\circ\|_\sex^{2-p}-c\frac{2-p}2t\right)^{\frac2{2-p}}$ which exists if $t\le\frac2{c(2-p)}\|u_\circ\|_\sex^2=:T_{ex}$.\par
We want to prove that $w(t)\le z(t)$ for any $t$. By contradiction, let us suppose that there exists $\bar t<T_{ex}$ and $\delta>0$ such that $w(\bar t)=z(\bar t)$ and $w(t)>z(t)$ for any $t\in (\bar t,\bar t+\delta)$ (we remind that $u\in C([0,T];L^2(\Omega)$).
By integration we get
\be\label{integralz}z(t)+c\int_{\bar t}^tz(\tau)^{\frac p2}\,d\tau=z(\bar t).\ee
Writing \eqref{integrineqw} with $\ve=\bar t$ and subtracting from it identity \eqref{integralz} we have
$$w(t)-z(t)+c\int_{\bar t}^t w(\tau)^{\frac p2}-z(\tau)^{\frac p2}\,d\tau\le w(\bar t)-z(\bar t)=0,\qquad \forall t\in(\bar t,\bar t+\delta)$$
which is impossible since $w(t)>z(t)$.
This concludes the proof in the case of finite time extinction.
\par
The uniqueness of the solution if $u_\circ\in L^\sex(\O)\cap L^2(\O)$ follows by Theorem \ref{existenceexterior}.
\par
Let us now consider the case $p=\frac{2n}{n+1}$ and $u_\circ\in L^1(\O)\cap L^s(\O)$ with $1<s\le2$. We remark that in this case $\sex=1$ and, as before, we consider the solution $u$ provided by Theorem \ref{Lsexistence}. With this choice of exponents, inequality \eqref{liminf} becomes
\be\label{L1normdecreas}\|u(t)\|_1\le\|u_\circ\|_1.\ee
Substituting $\psi$ with $u$ in equation \eqref{differential2.5} we get
\be\label{energyL2}\frac d{dt}\|u(t)\|_2^2+2\|\nabla u(t)\|_{\frac{2n}{n+1}}^{\frac{2n}{n+1}}=0\qquad\forall t\in(\ve,T).\ee
By means of the Gagliardo-Nirenberg inequality we have
$$\|u(t)\|_2\le c\|\nabla u(t)\|_{\frac{2n}{n+1}}^{\frac n{n+1}}\|u(t)\|_1^{\frac1{n+1}}$$
hence, by \eqref{L1normdecreas}, we obtain
$$\|\nabla u(t)\|_{\frac{2n}{n+1}}^{\frac{2n}{n+1}}\ge c\,\frac{\|u(t)\|_2^2}{\hskip.1cm\|u(t)\|_1^{\frac2{n+1}}}\ge c\,\frac{\|u(t)\|_2^2}{\|u_\circ\|_1^{\frac2{n+1}}}.$$
Substituting the above estimate in identity \eqref{energyL2} we get the differential inequality
$$\frac d{dt}\|u(t)\|_2^2+\frac c{\|u_\circ\|_1^{\frac2{n+1}}}\|u(t)\|_2^2\le0\qquad\forall t\in(\ve,T)$$
which gives
$$\frac d{dt}\log\left(\|u(t)\|_2^2\right)\le-\frac{c}{\|u_\circ\|_1^{\frac2{n+1}}\hskip-.4cm}$$
and, integrating on $(\ve,t)$
$$\|u(t)\|_2^2\le\|u(\ve)\|_2^2 \,e^{-c(t-\ve)\|u_\circ\|_1^{-2/(n+1)}}.$$
Thanks to \eqref{L2leqLs} we finally get
$$\|u(t)\|_2\le \frac c{\ve^\gamma}\|u_\circ\|_s^\alpha e^{-c(t-\ve)\|u_\circ\|_1^{-1/(n+1)}}\qquad\forall t>\ve.$$
\end{proof}

\section{The energy relation for linear parabolic systems: an extension to $L^q$ norm, $q\in(1,2]$}\label{energyLq}

In this last section we prove Theorem \ref{LPS}.

\begin{proof}
The existence and uniqueness for this problem is a classical result. To obtain the estimate \eqref{LPS-I} we multiply equation \eqref{LPP}$_1$ by $u(|u|^2+\delta)^{\frac{\sigma-1}2}h_R$, with $h_R$ defined in \eqref{hR}. 
We remark that, 
%in virtue of the uniqueness, the solution enjoys all the properties of the one constructed in Theorem \ref{Lsexistence}. 
Theorem \ref{EXTIN} is stated for $p<2$ but the computations in its proof make perfectly sense also for $p=2$, since the existence is known.
Hence we can proceed as in the proof of Theorem \ref{EXTIN}, substituting $p=2$ to get \eqref{usigmau0sigma} that becomes
$$\|u(t)\|_\sigma^2+2(\sigma-1)\int_0^t\|\nabla u\|_\sigma^2\,d\tau\le\|u_\circ\|_\sigma^2,\qquad\forall t>0.$$

\end{proof}

\vskip2cm
\leftskip5.4cm
\noindent 
Dipartimento di Matematica e Fisica\\
Universit\`a degli Studi della Campania ``L. Vanvitelli''\\
Viale Lincoln 5, 81100 Caserta,  Italy\\
francesca.crispo@unicampania.it
\bigskip
\par
\noindent
Dipartimento di Matematica\\
Universit\`a di Pisa\\
Via Buonarroti 1, 56127 Pisa, Italy\\
carlo.romano.grisanti@unipi.it
\bigskip
\par
\noindent
Dipartimento di Matematica e Fisica\\
Universit\`a degli Studi della Campania ``L. Vanvitelli''\\
Viale Lincoln 5, 81100 Caserta,  Italy\\
paolo.maremonti@unicampania.it

\end{document}